\documentclass[11pt, a4paper]{amsart}
\usepackage{fullpage}
\usepackage{amscd}
\usepackage{amsmath}
\usepackage{amsthm}
\usepackage{amssymb}
\usepackage{subfig}
\usepackage{graphicx}
\usepackage{mathtools}
\usepackage[colorlinks,linkcolor=blue,filecolor=blue,citecolor=blue,urlcolor=blue,bookmarks=false,hypertexnames]{hyperref}
\usepackage{tikz}
\usepackage{caption}
\usepackage{easyReview}

\newcommand{\eqdef}{\mathrel{\mathop:}=}

\newcommand{\cC}{\mathcal{C}}

\newcommand{\cP}{\mathcal{P}}
\newcommand{\cH}{\mathcal{H}}

\newcommand{\cB}{\mathcal{B}}

\newcommand{\cW}{\mathcal{W}}

\newcommand{\cS}{\mathcal{S}}

\newcommand{\cV}{\mathcal{V}}
\newcommand{\cG}{\mathcal{G}}

\def\MC{{\mathcal C}}
\def\MP{{\mathcal P}}

\def\MW{{\mathcal W}}
\def\MC{{\mathcal C}}
\newcommand{\numberset}{\mathbb}

\newcommand{\Z}{\numberset{Z}}

\newcommand{\R}{\numberset{R}}

\def\ZZ{{\mathbb Z}}

\def\pp{{\mathfrak p}}
\def\opn#1#2{\def#1{\operatorname{#2}}} 

\def\opn#1#2{\def#1{\operatorname{#2}}} 
\opn\height{height}

\makeatletter
\@addtoreset{equation}{section}
\def\theequation{\thesection.\@arabic \c@equation}
\makeatother

\theoremstyle{plain}

\newtheorem{theorem}[equation]{Theorem}

\newtheorem{notation}[equation]{Notation}

\newtheorem{Lemma}[equation]{Lemma}
\newtheorem{proposition}[equation]{Proposition}

\theoremstyle{definition}

\newtheorem{definition}[equation]{Definition}

\newtheorem{discussion}[equation]{Discussion}
\newenvironment{discussionbox}[1][]{%
\begin{discussion}[#1]\pushQED{\qed}}{\popQED \end{discussion}}
\newtheorem{example}[equation]{Example}

\newtheorem{remark}[equation]{Remark}

\makeatletter
\@namedef{subjclassname@2020}{\textup{2020} Mathematics Subject Classification}
\makeatother

\title{Polyocollection ideals and primary decomposition of polyomino ideals}

                \author{Carmelo Cisto}
                \address{Dipartimento di Scienze Matematiche e Informatiche, Scienze Fisiche e Scienze della Terra, Università di Messina, Viale Ferdinando Stagno D’Alcontres 31, Messina, 98166, Italy.}
                \email{carmelo.cisto@unime.it}
                
                \author{Francesco Navarra}
                \address{Dipartimento di Scienze Matematiche e Informatiche, Scienze Fisiche e Scienze della Terra, Università di Messina, Viale Ferdinando Stagno D’Alcontres 31, Messina, 98166, Italy.}
                \email{francesco.navarra@unime.it}

                \author{Dharm Veer}
                \address{Indian Institute of Technology Gandhinagar, Palaj, Gujarat 382355. India
                \newline 
            Chennai Mathematical Institute, Siruseri, Tamilnadu 603103. India}
                \email{dharm.v@iitgn.ac.in}

\thanks{}

\subjclass[2020]{05B50, 05E40}

\keywords{Polyominoes, primary decomposition, zig-zag walk.}

\begin{document}

\begin{abstract}
In this article, we study the primary decomposition of some binomial ideals. In particular, we introduce the concept of \textit{polyocollection}, a combinatorial object that generalizes the definitions of collection of cells and polyomino, that can be used to compute a primary decomposition of non-prime polyomino ideals. Furthermore, we give a description of the minimal primary decomposition of non-prime closed path polyominoes. In particular, for such a class of polyominoes, we characterize the set of all zig-zag walks and show that the minimal prime ideals have a very nice combinatorial description.   
\end{abstract}

\maketitle
\section{Introduction}
\noindent A polyomino is a finite collection of unit squares with vertices at lattice points in the plane joined edge by edge.
Qureshi~\cite{QUR12} associated a finite type $K$-algebra $K[\MP]$ (over a field $K$) to a polyomino $\MP$. 
An interesting question here is to relate the algebraic properties of the $K$-algebra $K[\MP]$ with the combinatorial properties of the polyomino $\MP$.
For certain classes of polyominoes, some algebraic properties of the associated algebra have been studied. For example, 
the graded minimal free resolution of $K[\MP]$ for convex polyominoes has been studied in~\cite{EHHlinearlyrelated, GMgreen-lazarsfeld}
and the Hilbert series of $K[\MP]$ has been studied in~\cite{thin, KVh-polynomial, hilbert_parallelogram, KVh-Charney-Davies, hilbert-closed-path}.\\
\noindent We say that a polyomino $\MP$ is prime (resp. non-prime) if the defining ideal of $K[\MP]$, denoted by $I_\MP$, is prime (resp. non-prime).
The ideal $I_\MP$ is called the {\em polyomino ideal}.
It is known that simple polyominoes are prime~\cite{coordinate, QSS}.
Hibi-Qureshi~\cite{HQ15non-simple} and Shikama~\cite{Shikama18nonsimple} proved that the non-simple polyominoes obtained by removing a convex polyomino from a rectangle are prime.
Characterizing prime polyominoes is a difficult task;
to do that Mascia et.~al.~\cite{MRR20primalityOfPolyominoes} defined zig-zag walks (see Definition~\ref{def:zig-zagwalk}) for polyominoes.
They proved~\cite[Corollary\ 3.6]{MRR20primalityOfPolyominoes} that the existence of zig-zag walks determines the non-primality of the polyomino
and conjectured that~\cite[Conjecture\ 4.6]{MRR20primalityOfPolyominoes} this property characterizes non-prime polyominoes.
This conjecture has been verified for certain classes of polyominoes, for example, for grid polyominoes~\cite{MRR20primalityOfPolyominoes}, for closed path polyominoes~\cite{Cisto_Navarra_closed_path} and for weakly closed path polyominoes~\cite{weakly}.
In~\cite{MRR22primalityOfPolyominoes}, Mascia et.~al. have determined some prime polyominoes using Gr\"{o}bner basis.\\
\noindent Concerning other algebraic properties, several mathematicians have studied the combinatorial description of the height of polyomino ideals. For some classes of polyominoes, it has been proved that the height of the polyomino ideal $I_{\MP}$ is the number of cells of $\MP$.
For instance, Qureshi~\cite{QUR12} proved this for convex polyominoes. 
Herzog-Madani~\cite{coordinate} extended this to simple polyominoes.
Extending this further, Herzog et.~al.~\cite{height} recently proved this for the case when the polyomino ideal is unmixed.
Dinu and the second author of this article~\cite{konig_type} proved this for closed path polyominoes and conjectured that this holds for all polyomino ideals. In this regard, we show that if removing a cell makes the polyomino simple, then the above conjecture holds (Proposition~\ref{prop:I_Pheight}).

\noindent In this article, we provide some properties about primary decomposition of polyomino ideals and we study the minimal primary decomposition of non-prime closed path polyominoes. The primary decomposition have been studied for some classes of binomial ideals, for instance in \cite{binomial_edge} for binomial edge ideals and in \cite{adjacent_minors} for ideals generated by adjacent 2-minors. The first results explained in this paper allows to consider something similar for polyomino ideals. In particular, supported by some examples obtained by using the software \texttt{Macaulay2} (\cite{macaulay2}) and inspired by the concept of \emph{admissible set} in \cite{adjacent_minors}, we find that for studying the primary decomposition of polyomino ideals one can consider a larger class of binomial ideals. This class is related to a combinatorial object that generalizes the concept of collection of cells and polyominoes. We call such a class of combinatorial objects \emph{polyocollections}. For non-prime closed path polyominoes we provide, in our main result, a detailed description of the minimal primary decomposition. 
We show that the polyomino ideal is the intersection of two minimal prime ideals (Theorem~\ref{thm:mainprimary}), and both minimal prime ideals have a very nice combinatorial description.
One of the minimal primes is the toric ideal appeared in~\cite[Section\ 3]{MRR20primalityOfPolyominoes} (see Theorem~\ref{thm:p_1prime}). 
We characterize zig-zag walks of non-prime closed path polyominoes and use that to define the other minimal prime which is generated by monomials and binomials. Finally, we also show that the height of these two ideals is equal to the number of the cells of the polyomino, and as a consequence the polyomino ideal is unmixed.\\ 
\noindent The article is organised as follows. In Section~\ref{sec:preliminaries}, we recall some basic notions on combinatorics related to polyominoes. In Section~\ref{sec:polycollection}, we define polyocollections and we associate a binomial ideal to it, generalizing the ideal associated to a collection of cells in \cite{QUR12}. Moreover, we provide a characterization of the primality of that binomial ideal in terms of a lattice ideal attached to the  polyocollection and we give a primary decomposition of the radical of the ideal associated to the polyocollection.  Section~\ref{sec:closedpath} is devoted to study the minimal primary decomposition of closed path polyominoes having zig-zag walks, equivalently the polyomino is non-prime (see  \cite{Cisto_Navarra_closed_path}). In the last section, we highlight some possible future directions and open questions about polyominoes and polyocollection ideals.

\section{Intervals, cells and polyominoes}\label{sec:preliminaries}

\noindent Let $(i,j),(k,l)\in \Z^2$. We say that $(i,j)\leq(k,l)$ if $i\leq k$ and $j\leq l$. Consider $a=(i,j)$ and $b=(k,l)$ in $\Z^2$ with $a\leq b$. The set $[a,b]=\{(m,n)\in \Z^2: i\leq m\leq k,\ j\leq n\leq l \}$ is called an \textit{interval} of $\Z^2$. 
In addition, if $i< k$ and $j<l$ then $[a,b]$ is a \textit{proper} interval. If $j=l$ (or $i=k$) then $a$ and $b$ are in \textit{horizontal} (or \textit{vertical}) \textit{position}. We also denote $]a,b[=\{(m,n)\in \Z^2: i< m< k,\ j< n< l\}$ and  $]a,b]$ or $[a,b[$ with the usual meaning. We also need to consider intervals in $\mathbb{R}^2$, for which we use the notations $\operatorname{c}([a,b])=\{(m,n)\in \mathbb{R}^2: i\leq m\leq k,\ j\leq n\leq l \}$, that is the \emph{closure} of $[a,b]$ in $\mathbb{R}^2$, and $\mathrm{int}([a,b])=\{(r,s)\in \R^2: i< r< k,\ j< s<l \}$, that is the \emph{interior} of $[a,b]$ in $\mathbb{R}^2$.\\
 Suppose that $[a,b]$ is a proper interval. In such a case we say $a, b$ the \textit{diagonal corners} of $[a,b]$ and $c=(i,l)$, $d=(k,j)$ the \textit{anti-diagonal corners} of $[a,b]$. Moreover, the elements $a$, $b$, $c$ and $d$ are called respectively the \textit{lower left}, \textit{upper right}, \textit{upper left} and \textit{lower right} \textit{corners} of $[a,b]$. We also call $V([a,b])=\{a,b,c,d\}$ the set of \emph{vertices} of the interval, while the sets $[a,c],[a,d],[b,d],[b,c]$ are the \textit{edges} of the interval and we denote $E([a,b])=\{[a,c],[a,d],[b,d],[b,c]\}$. If $[a,c]$ is an edge of an interval, we call the set $\{a,c\}$ the \emph{boundary of the edge}. \\
Let $\cS$ be a non-empty collection of intervals in $\Z^2$. The set of the vertices and of the edges of $\cS$ are respectively $V(\cS)=\bigcup_{I\in \cS}V(I)$ and $E(\cS)=\bigcup_{I\in \cS}E(I)$. We denote by $|\cS|$ the number of intervals belonging to the collection $\cS$, called also the \emph{rank} of $\cS$.\\ 
A proper interval $C=[a,b]$ with $b=a+(1,1)$ is called a \textit{cell} of $\ZZ^2$. Consider now $\cS$ be a collection of cells. 
If $C$ and $D$ are two distinct cells of $\cS$, then a \textit{walk} from $C$ to $D$ in $\cS$ is a sequence $\cC:C=C_1,\dots,C_m=D$ of cells of $\ZZ^2$ such that $C_i \cap C_{i+1}$ is an edge of $C_i$ and $C_{i+1}$ for $i=1,\dots,m-1$. In addition, if $C_i \neq C_j$ for all $i\neq j$, then $\cC$ is called a \textit{path} from $C$ to $D$. We say that $C$ and $D$ are connected in $\cS$ if there exists a path of cells in $\cS$ from $C$ to $D$.
 If $\cP$ is a non-empty and finite collection of cells in $\Z^2$, then it is called a \emph{polyomino} if any two cells of $\cP$ are connected in $\cP$. For instance, see Figure \ref{Figura: Polimino introduzione}.
\begin{figure}[h]
	\centering
	\includegraphics[scale=0.6]{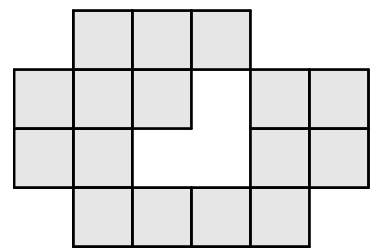}
	\caption{A polyomino.}
	\label{Figura: Polimino introduzione}
\end{figure}
 We say that a collection of cells $\cP$ is \textit{simple} if for any two cells $C$ and $D$ not in $\cP$ there exists a path of cells not in $\cP$ from $C$ to $D$. A finite collection of cells $\cH$ not in $\cP$ is a \textit{hole} of $\cP$ if any two cells of $\cH$ are connected in $\cH$ and $\cH$ is maximal with respect to set inclusion. For example, the polyomino in Figure \ref{Figura: Polimino introduzione} is not simple with an hole. Obviously, each hole of $\cP$ is a simple polyomino and $\cP$ is simple if and only if it has not any hole. 
Consider two cells $A$ and $B$ of $\Z^2$ with $a=(i,j)$ and $b=(k,l)$ as the lower left corners of $A$ and $B$ and $a\leq b$. A \textit{cell interval} $[A,B]$ is the set of the cells of $\Z^2$ with lower left corner $(r,s)$ such that $i\leqslant r\leqslant k$ and $j\leqslant s\leqslant l$. If $(i,j)$ and $(k,l)$ are in horizontal (or vertical) position, we say that the cells $A$ and $B$ are in \textit{horizontal} (or \textit{vertical}) \textit{position}. 
Let $\cP$ be a collection of cells. Consider  two cells $A$ and $B$ of $\cP$ in vertical or horizontal position. 	 
The cell interval $[A,B]$, containing $n>1$ cells, is called a
\textit{block of $\cP$ of rank n} if all cells of $[A,B]$ belong to $\cP$. The cells $A$ and $B$ are called \textit{extremal} cells of $[A,B]$. Moreover, a block $\cB$ of $\cP$ is \textit{maximal} if there does not exist any block of $\cP$ which contains properly $\cB$. 

\section{Polyocollections and related ideals}\label{sec:polycollection}

\noindent Let $\cC$ be a collection of intervals in $\mathbb{Z}^2$. We say that $\cC$ is a \emph{polyocollection} if for all $I,J\in \cC$ with $I\ne J$, we have $I\nsubseteq J$ and one of the following holds:
\begin{enumerate}
	\item $I\cap J$ is a common edge of $I$ and $J$. 
	\item For all $F\in E(I)$ and for all $G\in E(J)$, $|F\cap G|\leq 1$.  
\end{enumerate}

\begin{example} \rm
	$\cC_1=\{[(1,1),(2,2)],[(1,2),(2,3)],[(2,1),(5,2)],[(2,2),(5,3)],[(5,3),[7,4])]\}$ is a polyocollection, it is pictured in Figure~\ref{img1-1}.\\
	The collection $\cC_2=\{[(3,1),(4,2)],[(4,1),(5,2)],[(3,2),(4,6)],[(4,2),(5,6)],[(3,6),(4,7)],\newline [(4,6),(5,7)],[(1,3),(2,4)],[(1,4),(2,5)],[(2,3),(6,4)],[(2,4),(6,5)],[(6,3),(7,4)],[(6,4),(7,5)]\}$, displayed in Figure~\ref{img1-2}, is also a polyocollection.\\
	The collection $\cC_3=\{[(1,2),(3,4)],[(2,1),(4,3)],[(4,1),(5,2)],[(4,2),(5,3)]\}$, displayed in Figure~\ref{img1-3}, is not a polyocollection. Because if we take $I=[(2,1),(4,3)]$ and $J=[(4,1),(5,2)]$, then $I$ and $J$ does not satisfy any of the two conditions for a polyocollection (the same hold also for $I=[(2,1),(4,3)]$ and $J=[(4,2),(5,3)]$). Observe that, for $\cC_3$, the position of $[(1,2),(3,4)]$ with respect to $[(2,1),(4,3)]$ is not in contradiction with the definition of polyocollection.\\	
	The collection $\cC_4=\{[(1,1),(3,3)],[(1,3),(3,5)],[(3,1),(5,3)],[(3,3),(5,5)],[(2,2),(4,4)]\}$, displayed in Figure~\ref{img1-4}, is another example of polyocollection.
	
	\begin{figure}[h!]
		\subfloat[]{\includegraphics[scale=0.6]{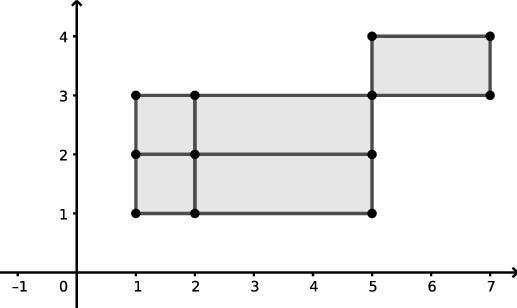}\label{img1-1}}\qquad \qquad
		\subfloat[]{\includegraphics[scale=0.6]{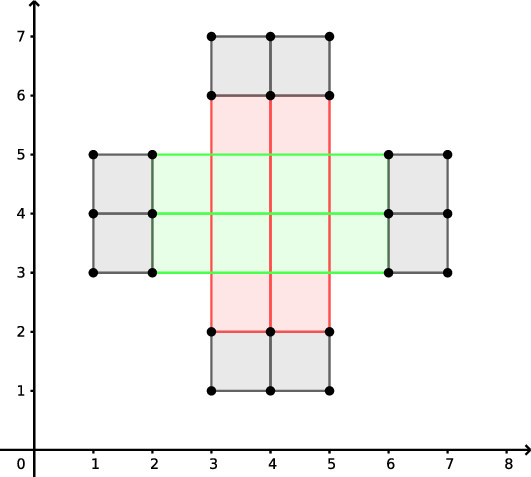}\label{img1-2}}\qquad \qquad
		\subfloat[]{\includegraphics[scale=0.6]{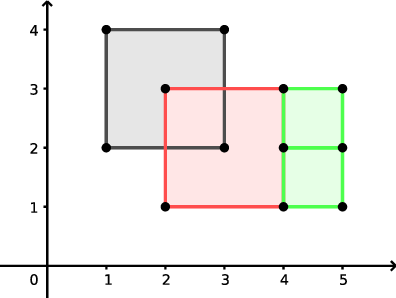}\label{img1-3}}\qquad \qquad
		\subfloat[]{\includegraphics[scale=0.6]{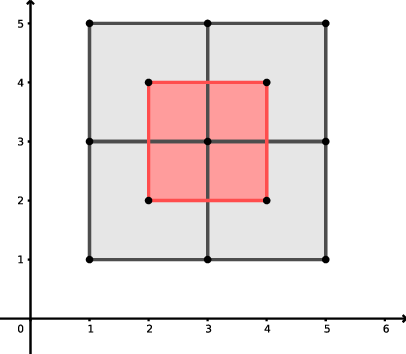}\label{img1-4}}
		\caption{(a), (b) and (d) are polyocollections, while (c) is not a polyocollection.}
		\label{img1}
	\end{figure}
\label{example1}	
\end{example}

     An interval $I$ of $\mathbb{Z}^2$ is called an \emph{inner interval of $\cC$} if there exist $n\in \mathbb{N}$ and $I_1,I_2,\ldots,I_n\in \cC$ such that $\operatorname{c}(I)=\bigcup_{i=1}^n \operatorname{c}(I_i)$, with  $\mathrm{int}(I_{k})\cap \mathrm{int}(I_{l})=\emptyset$ for all $k\neq l$.

We denote by $G(\cC)$ the set of inner intervals of $\cC$. By convention we consider that the empty set is a polyocollection whose set of inner interval is the empty set.\\
For instance, referring to Example~\ref{example1}, the intervals $[(1,1),(2,3)]$ and $[(1,1),(5,3)]$ of $\cC_1$ are examples of inner intervals. In $\cC_2$, $[(1,3),(7,5)]$ and $[(3,2),(5,7)]$ are examples of inner intervals, while $[(3,3),(5,5)]$ is not an inner interval.

	\noindent Let $\cC$ be a polyocollection. Define the polynomial ring $S_\cC=K[x_a\mid a \in V(\cC)]$, with $K$ a field. If $I=[a,b]$ is an inner interval of $\cC$, having anti-diagonal corners $c,d$, then we call the binomial $f_I=x_a x_b - x_c x_d \in S_\cC$ an \emph{inner 2-minor} of $\cC$. We define by $I_\cC$ the ideal generated by the binomials $f_I$ for all inner intervals $I$ of $\cC$ and we call it the \textit{polyocollection ideal} of $\cC$. The quotient ring $K[\cC]=S_\cC/I_\cC$ is said the \emph{coordinate ring} of $\cC$.
    For a monomial $u\in S_\cC$, we define the {\em support} of $u$, denoted by $\mathrm{supp}(u)$ to be the set of variables $x_r\in S_\cC$ such that $x_r$ divides $u$ in $S_\cC$.

\begin{remark}
Let $\cP$ be a collection of cells. Observe that $\cP$ is trivially a polyocollection. In such a case the definition of inner intervals reduces to the following: an interval $I\in \mathbb{Z}^2$ is an inner interval of $\cP$ if and only if each cell contained in $I$ belongs $\cP$. Moreover the ideal $I_\cP$ as defined above for polyocollections coincides with the ideal associated to a collection of cells following the definition in \cite{QUR12}.
\end{remark}

\noindent If $\cC$ is a polyocollection and $\cP$ is a collection of cells, we say that $\cC$ and $\cP$ are \textit{algebraically isomorphic} if $K[\cC]$ is isomorphic to $K[\cP]$ as $K$-algebras. With reference to Example \ref{example1}, note that $\cC_1$, $\cC_2$ and $\cC_4$ are algebraically isomorphic to the collections of cells shown in Figures \ref{fig:polyo_coll_1}, \ref{fig:polyo_coll_2} and \ref{fig:polyo_coll_3} respectively.  

 \begin{figure}[h!]
	\subfloat[]{\includegraphics[scale=0.6]{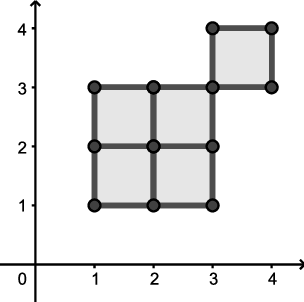}\label{fig:polyo_coll_1}}\qquad \qquad
	\subfloat[]{\includegraphics[scale=0.6]{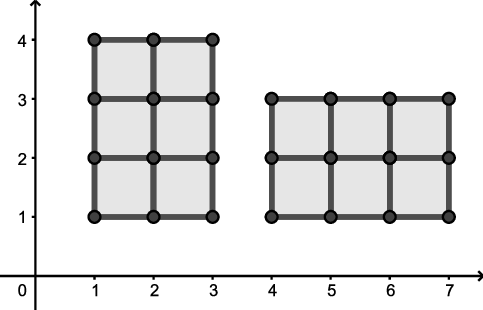}\label{fig:polyo_coll_2}}\qquad \qquad
	\subfloat[]{\includegraphics[scale=0.6]{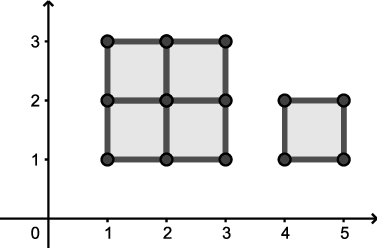}\label{fig:polyo_coll_3}}
	\caption{The collections of cells which are algebraically isomorphic to the polyocollections of Figures \ref{img1-1}, \ref{img1-2} and \ref{img1-4}, respectively.}
	\label{Figure: polyo iso coll}
\end{figure} 

In the following discussion, we prove that there exists a polyocollection $\cC$ which is not algebraically isomorphic to any collection of cells. Roughly speaking, that means the class of polyocollections is algebraically strictly bigger than the class of collections of cells. 

\begin{discussionbox}\label{dis:non-iso}
    Let $\MC$ be the polyocollection as shown in Figure~\ref{fig:polyo}. 
    Then, \[K[\MC] = K[x_{21},x_{41},x_{12},x_{22},x_{42},x_{52},
    x_{13},x_{23},x_{33},x_{43},x_{53},x_{24},x_{34},x_{44}]\] 
    and the polyocollection ideal 
    \begin{align*}
    I_\MC = & (x_{42}x_{53}-x_{43}x_{52},\,x_{33}x_{44}-x_{34}x_{43},\, x_{23}x_{44}-x_{24}x_{43},\,x_{23}x_{34}-x_{24}x_{33},\\
    & x_{21}x_{42}-x_{22}x_{41},\,x_{12}x_{23}-x_{13}x_{22}). 
  \end{align*}
    \texttt{Macaulay2} computations show that $I_\MC$ is a non-prime ideal of height 5 and the minimal primary decomposition of the ideal $I_\MC$ consists of two prime ideals $\mathfrak{p}_1$ and $\mathfrak{p}_2$, that are the following:
    \begin{align*}
    \mathfrak{p}_1  = & 
   I_\MC+ (x_{12}x_{24}x_{41}x_{53}-x_{13}x_{21}x_{44}x_{52})
\\
 \mathfrak{p}_2= &(x_{43},x_{42},x_{33},x_{23},x_{22})
 \end{align*}
    The height of both ideals in the primary decomposition is 5. So, $I_\MC$ is a unmixed ideal.

    Suppose that, if possible, there exists a collection of cells $\MP$ such that $K[\MC]$ is isomorphic to $K[\MP]$.
    Then, by~\cite[Theorem 3.1]{height}, $\MP$ is a collection of 5 cells. 
    Since $I_\MP$ is a non-prime ideal, $\MP$ is a non-simple collection of cells by~\cite[Theorem 3.3]{weakly}.
    Therefore, up to rotations and reflections, the only possible collections of cells $\MP$ are those in Figures~\ref{fig:poyomino1}-\ref{fig:poyomino4}.
    For $\MP$ as in Figure~\ref{fig:poyomino2} and~\ref{fig:poyomino3}, \texttt{Macaulay2} computations show that the dimension of the ring $K[\MP]$ is $10$ and $11$ respectively. 
    As the dimension of the ring $K[\MC]$ is 9, we get that $K[\MC]$ is not isomorphic to $K[\MP]$ when $\MP$ is as in Figure~\ref{fig:poyomino2} and~\ref{fig:poyomino3}. 
    For $\MP$ as in Figure~\ref{fig:poyomino4}, $K[\cP]$ is a domain (see~\cite[Remark\ 3.4]{weakly}), so it is not isomorphic to $K[\cC]$.
When $\MP$ is as in Figure~\ref{fig:poyomino1}, the minimal primary decomposition of the ideal $I_\MP$ consists of two prime ideals $\mathfrak{q}_1$ and $\mathfrak{q}_2$, that are the following: 
  \begin{align*}
    \mathfrak{q}_1  = & 
I_\MP + (x_{12}x_{25}x_{31}x_{43}-x_{13}x_{21}x_{35}x_{42},\,x_{12}x_{24}x_{31}x_{43}-x_{13}x_{21}x_{34}x_{42}) \\ 
    \mathfrak{q}_2 = & (x_{33},\,x_{32},\,x_{23},\,x_{22},\,x_{24}x_{35}-x_{25}x_{34})
    \end{align*}   
 One can immediately see that there is a monomial ideal generated by variables in the primary decomposition of $I_\MC$ but none of the ideals in the primary decomposition of $I_\MP$ is generated by monomials. 
 Hence $K[\MC]$ is not isomorphic to $K[\MP]$.
 \end{discussionbox}
 
 \begin{figure}[h!]
		\subfloat[]{\includegraphics[scale=0.6]{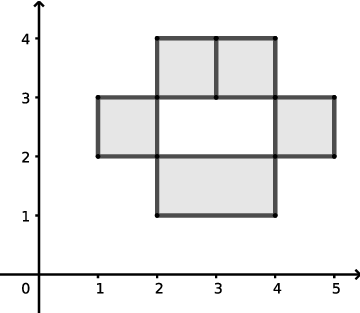}\label{fig:polyo}}\qquad \qquad
		\subfloat[]{\includegraphics[scale=0.6]{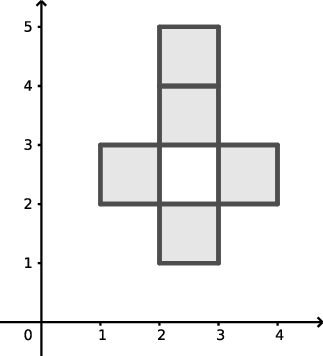}\label{fig:poyomino1}}\qquad \qquad
		\subfloat[]{\includegraphics[scale=0.6]{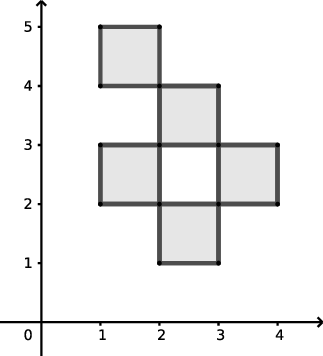}\label{fig:poyomino2}}\qquad \qquad
		\subfloat[]{\includegraphics[scale=0.6]{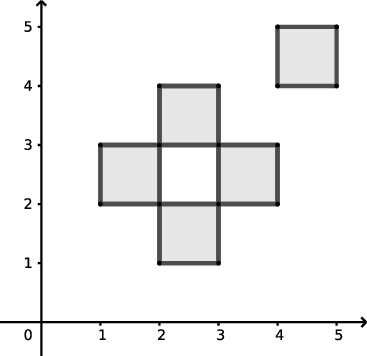}\label{fig:poyomino3}}\qquad \qquad
		\subfloat[]{\includegraphics[scale=0.6]{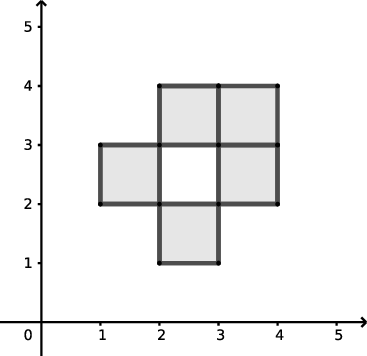}\label{fig:poyomino4}}
		\caption{The polyocollection and collections of cells examined in Discussion~\ref{dis:non-iso}.}
		\label{6inner}
	\end{figure} 

\noindent We want to show that some concepts and results introduced for collection of cells in \cite{QUR12}, can be extended in a natural way for polyocollections.

\subsection{The lattice ideal related to a polyocollections}
Let $\cC$ be a polyocollection. For $a \in V(\cC)$ denote with $\mathbf{v}_a$ the vector in $\mathbb{Z}^{|V(\cC)|}$ having $1$ in the coordinate indexed by $a$ and $0$ in all other coordinates. If $[a,b]\in \cC$ is an interval having $a,b$ as diagonal corners, and $c,d$ as anti-diagonal corners, we denote $\mathbf{v}_{[a,b]}=\mathbf{v}_a+ \mathbf{v}_b -\mathbf{v}_c -\mathbf{v}_d\in \mathbb{Z}^{|V(\cC)|}$.
We define $\Lambda_\cC$ the sub-lattice of $\mathbb{Z}^{|V(\cC)|}$ generated by the vectors $\mathbf{v}_I$ for all $I\in \cC$.  \\
Let $n=|V(\cC)|$, we recall some known notations. If $\mathbf{v}\in \mathbb{N}^n$ we denote as usual with $x^\mathbf{v}$ the monomial in $S_\cC$ having $\mathbf{v}$ as exponent vector. If $\mathbf{e}\in \mathbb{Z}^n$ we denote with $\mathbf{e}^+$ the vector obtained from $\mathbf{e}$ by replacing all negative components by zero, and $\mathbf{e}^-=-(\mathbf{e}-\mathbf{e}^+)$.\\
Let $L_\cC$ be the lattice ideal of $\Lambda_\cC$, that is the following binomial ideal in $S_\cC$:

$$L_\cC=(\{x^{\mathbf{e}^+}- x^{\mathbf{e}^-}\mid \mathbf{e}\in \Lambda_\cC\})$$


\noindent In \cite[Theorem 3.5]{QUR12} it is shown that if $\cP$ is a collection of cells then $L_\cP$ is a prime ideal. Using the same argument, we can prove that the same holds considering also polyocollections. We provide the proof of this fact for completeness.\\

\noindent Denote by $F(\cC)$ the set of elements in $V(\cC)$ that are not lower left corner of an interval in $\cC$.\\

\begin{Lemma}
	Let $\cC$ be a polyocollection and $\cB=\{\mathbf{v}_I \mid I\in \cC\} \cup \{\mathbf{v}_a\mid a\in F(\cC)\}$. Then $\cB$ is a basis of the group $\mathbb{Z}^{|V(\cC)|}$.
	\label{basis}
\end{Lemma}
\begin{proof}
    Let $n=|V(\cC)|$ and $M$ be the $n\times n$ matrix whose column vectors are $\cB=\{\mathbf{v}_I \mid I\in \cC\} \cup \{\mathbf{v}_a\mid a\in F(\cC)\}$. We refer to $M$ as the matrix associated to $\cC$. Observe that the columns of $M$ can be labelled with the set $\{I\in \cC\}\cup \{a\in F(\cC)\}$ and the rows can be labelled with the set $V(\cC)$. In order to prove our result we show that $\det(M)\neq 0$, by induction on $|\cC|$. If $|\cC|=1$ it is easy to check that $\det(M)=\pm 1$. Suppose $|\cC|>1$ and let $a\in V(\cC)$ be such that it is a minimal element under the natural partial order on $\ZZ^2$. Observe that in this case there exists an interval $J=[a,b]$ in $\cC$ such that $a\notin I$ for all $I\in \cC \setminus \{J\}$. Denote by $c,d$ the anti-diagonal corner of $J$. We distinguish three cases. 
    \begin{enumerate}
        \item[Case 1]   
            Either $c\in F(\cC)$ or $d\in F(\cC)$. Without loss of generality, we consider $c\in F(\cC)$  and $d\notin F(\cC)$. Let $M'$ be the matrix obtained by removing the row labelled with $a$  and the column labelled by $J$. Observe that the column removed contains all $0$ except for $1$ in the position of $a$, so $\det(M)=\pm \det(M')$. Moreover, let $M''$ be the matrix obtained from $M'$ by removing the row labelled with $c$  and the column labelled by $c$. It is easy to verify that $\det(M')=\pm \det(M'')$, so $\det(M)=\pm \det(M'')$. Furthermore, it is not difficult to verify that $M''$ is the matrix associated to the polyocollection $\cC'=\cC\setminus \{J\}$, so by induction $\det(M'')\neq 0$, that is $\det(M)\neq 0$.
    \item[Case 2]
	Both $c\in F(\cC)$ and $d\in F(\cC)$. In this case we can use the same argument, considering $\det(M)=\pm \det(M''')$, where $M'''$ is the matrix associated to $\cC'=\cC\setminus \{J\}$ and it is obtained from $M$ by removing the rows labelled with $a,c,d$  and the columns labelled by $J,c,d$.
\item[Case 3]
	$c$ and $d$ are both the lower left corner of an interval in $\cC$. Also in this case we can use the same argument, considering $\det(M)=\pm \det(M')$, where $M'$ is the matrix associated to $\cC'=\cC\setminus \{J\}$ and it is obtained from $M$ by removing the row labelled with $a$ and the column labelled by $J$.
    \end{enumerate}
\end{proof}

\begin{theorem}\label{thm:lc_prime}
Let $\cC$ be a polyocollection. Then $L_\cC$ is a prime ideal.
\end{theorem}
\begin{proof}
    Let $n=|V(\cC)|$. By Lemma~\ref{basis}, every $\mathbf{v}\in \mathbb{N}^n$ can be expressed in an unique way as $\mathbb{Z}$-linear combination of the vectors in $\cB=\{\mathbf{v}_I \mid I\in \cC\} \cup \{\mathbf{v}_a\mid a\in F(\cC)\}$. 
    In particular, for all $a\in V(\cC)$, we have $\mathbf{v}_a=\sum_{I\in \cC}\lambda_I^{(a)}\mathbf{v}_I+\sum_{b\in F(\cC)}\mu_b^{(a)}\mathbf{v}_b$, with $\lambda_I^{(a)},\mu_b^{(a)}\in \mathbb{Z}$. We define the following map:
    $$\psi : S_\cC \longrightarrow K[\{y_b^{\pm 1}\mid b\in F(\cC)\}]\ \ \mbox{with}\ \ \ x_a \mapsto \prod_{b\in F(\cC)}y_b^{\mu_b^{(a)}}.$$
	It is not difficult to check that if $\mathbf{v}\in \mathbb{N}^n$ and $\mathbf{v}=\sum_{I\in \cC}\lambda_I^{(v)}\mathbf{v}_I+\sum_{b\in F(\cC)}\mu_b^{(v)}\mathbf{v}_b$, with $\lambda_I^{(v)},\mu_b^{(v)}\in \mathbb{Z}$, then $\psi(x^\mathbf{v})=\prod_{b\in F(\cC)}y_b^{\mu_b^{(v)}}$. Observe that $\ker \psi$ is a toric ideal, in particular it is prime. We show that $L_\cC=\ker \psi$.\\
	$\subseteq)$ Let $x^\mathbf{v}-x^\mathbf{w}\in L_\cC$, then $\mathbf{v}-\mathbf{w}\in \Lambda_\cC$. So $\mathbf{v}=\mathbf{w}+\sum_{I\in \cC}\alpha_I\mathbf{v}_I$, with $\alpha_I\in \mathbb{Z}$, and by construction the coefficients $\alpha_I$ of $\mathbf{v}_I$, for $I\in \cC$, does not contribute to the computation of $\psi(x^\mathbf{v})$. Therefore $\psi(x^\mathbf{v})=\psi(x^\mathbf{w})$.\\
	$\supseteq)$ Let $x^\mathbf{v}-x^\mathbf{w}$ be a minimal generator of $\ker \psi$. In particular, $\mathrm{supp}(x^\mathbf{v})\cap \mathrm{supp}(x^\mathbf{w})=\emptyset$, that is $\mathbf{e}=\mathbf{v}-\mathbf{w}\in \mathbb{Z}^n$ with $\mathbf{e}^+=\mathbf{v}$ and $\mathbf{e}^-=\mathbf{w}$. Moreover $\psi(x^\mathbf{v})=\psi(x^\mathbf{w})=\prod_{b\in F(\cC)}y_b^{\mu_b}$ with $\mu_b\in \mathbb{Z}$. This means that $\mathbf{v}=\sum_{I\in \cC}\lambda_I^{(v)}\mathbf{v}_I+\sum_{b\in F(\cC)}\mu_b^{(v)}\mathbf{v}_b$ and $\mathbf{w}=\sum_{I\in \cC}\lambda_I^{(w)}\mathbf{v}_I+\sum_{b\in F(\cC)}\mu_b^{(w)}\mathbf{v}_b$ with $\mu_b^{(v)}=\mu_b^{(w)}=\mu_b$ for each $b\in F(\cC)$. So $\mathbf{v}-\mathbf{w}=\sum_{I\in \cC}(\lambda_I^{(v)}-\lambda_I^{(w)})\mathbf{v}_I \in \Lambda_\cC$, in particular $x^\mathbf{v}-x^\mathbf{w}\in L_\cC$.
\end{proof}

\noindent By the previous results, observe also that $L_\cC$ does not contain monomials. In fact, since it is a prime ideal, if it contains monomials then it contains a variable $x_a$ for $a\in V(\cC)$. But all elements in $L_\cC$ have degree greater than 2. 

\begin{Lemma}
	Let $\cC$ be a polyocollection. Then there exists a monomial $u\in S_\cC$ such that $L_\cC=(I_\cC : u)$. In particular $I_\cC \subseteq L_\cC$.
	\label{lemma-Romeo}
\end{Lemma}

\begin{proof}
	The proof is exactly the same as that provided by \cite[Lemma 2.1]{MRR22primalityOfPolyominoes} for a collection of cells.
\end{proof}

\begin{theorem}
	Let $\cC$ be a polyocollection. Then $I_\cC$ is a prime ideal if and only if $I_\cC=L_\cC$.
	\label{lemma-primality}
\end{theorem}
\begin{proof}
    If $I_\cC=L_\cC$, then trivially $I_\cC$ is prime by Theorem~\ref{thm:lc_prime}. Suppose that $I_\cC$ is prime. By Lemma~\ref{lemma-Romeo} we have that $I_\cC \subseteq L_\cC$ and $L_\cC=(I_\cC : u)$ for some monomial $u\in S_\cC$. So, if $f\in L_\cC$ then $fu\in I_\cC$ and observe that $u\notin I_\cC$. Since $I_\cC$ is prime then $f\in I_\cC$, so we proved $I_\cC \supseteq L_\cC$.  
\end{proof}

\begin{remark} \rm
	Let $\cC$ be a polyocollection and $I,J\in\cC$. We say that $I$ and $J$ are connected if there exists a sequence $I_1,\ldots,I_m\in \cC$ with $I_1=I$ and $I_m=J$ such that $I_i\cap I_{i+1}$ is a common vertex or a common edge for $i\in\{1,\ldots,m\}$.  We can consider in $\cC$ the equivalence relation $\simeq$, where $I\simeq J$ if and only if $I$ is connected with $J$. Observe that all equivalence classes of $\simeq$ are polyocollection, that we call the \emph{connected components} of $\cC$. If $\cC_1,\ldots,\cC_n$ are the connected components of $\cC$, it is not difficult to see that $L_{\cC}=\sum_{i=1}^n L_{\cC_i}$.
	
\end{remark}

\subsection{Primary decomposition in polyocollections}

Let $\cC$ be a polyocollection. A subset $X\subseteq V(\cC)$ is called an \textit{admissible set} of $\cC$ if for every inner interval $I$ of $\cC$, $X\cap V(I)=\emptyset$ or $X\cap V(I)$ contains the boundary of an edge of $I$. Observe that $\emptyset$ and $V(\cC)$ are admissible sets.

For an admissible set $X$ of $\cC$, define the following set:
$$\cG^{(X)}=\{I\in G(\cC) \mid V(I)\cap X=\emptyset\}.$$ 

\begin{proposition}
	Let $\cC$ be a polyocollection and $X$ be an admissible set. Then there exists a polyocollection $\cC^{(X)}$ such that  $\cG^{(X)}$ is the set of inner intervals of $\cC^{(X)}$.
\end{proposition}

\begin{proof}
	Let $\cC^{(X)}$ be the set of minimal intervals in $\cG^{(X)}$ with respect to set inclusion. We prove that $\cC^{(X)}$ is a polyocollection. Let $I,J\in \cC^{(X)}$, by construction it is trivial that $I\nsubseteq J$. Suppose $I\cap J \notin E(I)\cap E(J)$ and there exist $F\in E(I)$ and $G\in E(J)$ such that $|F\cap G|>1$. Then $I,J\notin \cC$ and we can assume, without loss of generality, that there exist $[a,b]\in E(I)$ and $[x,y]\in E(J)$, in particular $a,b\in V(I)$ and $x,y\in V(J)$, such that $[a,b]$ and $[x,y]$ belong to the same line, $a\in[x,y]$ and $x\notin [a,b]$. Let $z\in V(J)$ such that $\{x,z\}$ is the boundary of an edge of $J$ (see for instance Figure~\ref{img:prop}). Since $I,J\in \cG^{(X)}$ and $X$ is an admissible set, then $a,x,z \notin X$.
\begin{figure}[h!]
		\subfloat[]{\includegraphics[scale=0.60]{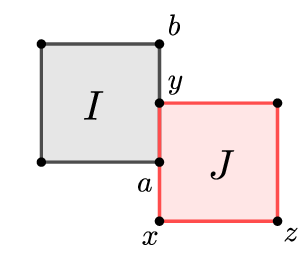}}\qquad \quad
		\subfloat[]{\includegraphics[scale=0.60]{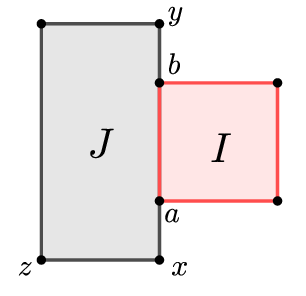}}\qquad \quad
		\caption{Two intervals $I,J$ such that $I\cap J \notin E(I)\cap E(J)$ and $|F\cap G|>1$ for some $F\in E(I)$ and $G\in E(J)$. }
		\label{img:prop}
	\end{figure}	
	
	 \noindent Moreover $a,x,z$ are the vertices of an inner interval $K$ of $\cC$, otherwise there exist two intervals $I_1, J_1\in \cC$, with $I_1\subseteq I$ and $J_1\subseteq J$, having no edges in common and two edges intersecting in two or more points, that contradicts $\cC$ is a polyocollection. In particular $K\subseteq J$ and $|X\cap V(K)|\leq 1$. Considering that $X$ is an admissible set, we have $X\cap V(K)=\emptyset$, obtaining $K\in \cG^{(X)}$ that contradicts the minimality of $J$. So $\cC^{(X)}$ is a polyocollection. 

     Finally, if $I\in \cG^{(X)}$ we show that $I$ is union of intervals in $\cC^{(X)}$ having pairwise disjoint common interior, that is $\cG^{(X)}$ is the set of inner intervals of the polyocollection $\cC^{(X)}$. So, put $I=[a,b]$ with $V(I)=\{a,b,c,d\}$. If $V(\cC)\cap (I\setminus V(I))\subseteq X$ then it is not difficult to see that $I\in \cC^{(X)}$, so suppose there exists $e\in V(\cC)\cap (I\setminus V(I))$ such that $e\notin X$. Consider two possibilities:  either $e$ belongs to an edge of $I$ or $e\in \mathrm{int}(I)$. 
     \begin{itemize}
         \item 
     Assume that $e$ belongs to an edge of $I$, say $[a,c]$ without loss of generality. Then it is not difficult to see that in such a case there exist $I_1,I_2$ inner intervals of $\cC$ with $\{a,e,d\}\subset V(I_1)$, $\{e,c,b\}\subset V(I_2)$ such that $\operatorname{c}(I)=\operatorname{c}(I_1) \cup \operatorname{c}(I_2)$. Moreover, $|X \cap I_1|\leq 1$ and $|X \cap I_2|\leq 1$ and since $X$ is an admissible set of $\cC$ we obtain $X \cap I_1=\emptyset$ and $X \cap I_2=\emptyset$, that is $I_1,I_2 \in \cG^{(X)}$. 
 \item
     Assume that $e\in \mathrm{int}(I)$. In such a case is not difficult to see that we can write $\operatorname{c}(I)=\operatorname{c}(I_1) \cup \operatorname{c}(I_2) \cup \operatorname{c}(I_3) \cup \operatorname{c}(I_4)$, with $\{a,e\}\subset V(I_1)$, $\{b,e\}\subset V(I_2)$, $\{c,e\}\subset V(I_3)$, $\{d,e\}\subset V(I_4)$ and $I_i$ is an inner interval of $\cC$ for $i\in \{1,2,3,4\}$. For all $i\in \{1,2,3,4\}$, observe that $X\cap V(I_i)$ does not contain the boundary of an edge of $I_i$, so  $X\cap V(I_i)=\emptyset$ considering that $X$ is an admissible set of $\cC$. Therefore $I_i\in \cG^{(X)}$ for all $i\in\{1,2,3,4\}$. 
     \end{itemize}
     It follows that in every case $I$ is union of intervals in $\mathbb{R}^2$ belonging to $\cG^{(X)}$ and having pairwise disjoint common interior. If such intervals are not minimal, with respect to set inclusion, we can repeat the previous arguments to these intervals obtaining the desired conclusion. 
\end{proof}

\noindent Denote by $\cC^{(X)}$ the polyocollection having $\cG^{(X)}$ as set of inner intervals. We define the following ideal
$$J_X=(\{x_a\mid a\in X\})+L_{\cC^{(X)}}\subseteq S_\cC.$$

\noindent It is not difficult to see that $I_\cC\subseteq J_X$ for all admissible sets $X$. In fact, by the previous result, $(\{x_a\mid a\in X\})+I_\cC=(\{x_a\mid a\in X\})+I_{\cC^{(X)}}\subseteq J_X$. Moreover, since the set of variables involved in $L_{\cC^{(X)}}$ is disjoint to the set $\{x_a\mid a\in X\}$, it is easy to see that $J_X$ is a prime ideal for all $X$ admissible sets of $\cC$.

\begin{remark} \rm 
    Let $\cC$ be a polyocollection and suppose that $I_\cC$ is not a prime ideal. Then $J_\emptyset=L_\cC$ is a minimal prime ideal of $I_\cC$; otherwise let $\mathfrak{p}$ be a prime ideal such that $I_\cC \subset \mathfrak{p} \subseteq L_\cC$. If $f\in L_\cC$, then there exists a monomial $u\in S_\cC$ such that $fu\in I_\cC \subset \mathfrak{p}$. If $u\in \mathfrak{p}$, since $\mathfrak{p}$ is prime there exists $x_a\in \mathrm{supp}(u)$ such that $x_a \in \mathfrak{p}$, but this is not possible since $L_\cC$ does not contain variables. So $u\notin \mathfrak{p}$, in particular $f\in \mathfrak{p}$. Therefore $\mathfrak{p}=L_\cC$.
	
\end{remark}

\begin{proposition}
    Let $\cC$ be a polyocollection and suppose that $I_\cC$ is not a prime ideal. Then for every minimal prime ideal $\mathfrak{p}$ of $I_\cC$, there exists an admissible set $X$ such that $\mathfrak{p}=J_X$
	\label{prop:minimal}
\end{proposition}
\begin{proof}
    Let $\mathfrak{p}$ be a minimal prime ideal of $I_\cC$ and let $Y=\{a \in V(\cC)\mid x_a \in \mathfrak{p}\}$. Note that $(\{x_a\mid a\in Y\})+I_\cC \subset \mathfrak{p}$. We show that $Y$ is an admissible set for $\cC$. In fact, if $Y$ is not an admissible set, then there exists an inner interval $I$ of $\cC$ such that $|Y\cap V(I)|=1$ or $Y\cap V(I)=\{a,b\}$ with $a,b$ both diagonal (or anti-diagonal) corners of $I$. So, considering the binomial $f_I=x_a x_b-x_c x_d$, we obtain that either $x_a x_b\in \mathfrak{p}$ with $a,b \notin Y$ or $x_c x_d\in \mathfrak{p}$ with $c,d \notin Y$, that is a contradiction for the primality of $\mathfrak{p}$. 

    Now, we can write $(\{x_a\mid a\in Y\})+I_\cC=(\{x_a\mid a\in Y\})+I_{\cC^{(Y)}}$ and it is contained in $\mathfrak{p}$. If $f\in L_{\cC^{(Y)}}$ then by Lemma~\ref{lemma-Romeo}, there exists a monomial $u\in S_{\cC^{(Y)}}$ such that $fu\in I_{\cC^{(Y)}} \subseteq \mathfrak{p}$, and in particular $\mathrm{supp}(u)\cap \{x_a\mid a\in Y\}=\emptyset$. As consequence $u\notin \mathfrak{p}$, otherwise from the primality of $\mathfrak{p}$ there exists a $x_a\in \mathrm{supp}(u)$ such that $x_a\in \mathfrak{p}$, obtaining $a\in Y$, that is a contradiction. Therefore $f\in \mathfrak{p}$, so $(\{x_a\mid a\in Y\})+ L_{\cC^{(Y)}}\subseteq \mathfrak{p}$ and by the minimality of $\mathfrak{p}$ we obtain $(\{x_a\mid a\in Y\})+ L_{\cC^{(Y)}}=\mathfrak{p}$.
\end{proof}

\begin{theorem}
	Let $\cC$ be a polyocollection. Then $\sqrt{ I_\cC}=\bigcap_{X} J_X$ where $X$ moves overall the admissible sets of $\cC$.
	\label{primary-total}
\end{theorem}
\begin{proof}
	We have $I_\cC\subseteq \bigcap_{X} J_X$ where $X$ moves overall the admissible sets of $\cC$, so $\sqrt{ I_\cC}\subseteq \bigcap_{X} J_X$. Moreover each minimal prime of $I_\cC$ is also a minimal prime of $\sqrt{ I_\cC}$. Finally, since a radical ideal is the intersection of its minimal prime
	ideals, the assertion of the theorem will follow. 
\end{proof}

\begin{remark}
If $\cC$ is a polyocollection and $I_\cC$ is radical, then Theorem~\ref{primary-total} provides a primary decomposition of it. Such decomposition is not minimal since the converse of Proposition~\ref{prop:minimal} does not hold in general. We will provide an example of this fact in the last section. 
\end{remark}

\section{Primary decomposition of closed path polyominoes}\label{sec:closedpath}

\noindent In this section, we consider a particular class of polyominoes, called \emph{closed paths}, for which we provide the minimal primary decomposition of the polyomino ideal. We already observed that a collection of cells (or a polyomino in particular) is also a polyocollection, but with respect to the previous section we give a more detailed (and combinatorial) description of the primary decomposition, taking advantage of the structure of closed path. 



\subsection{Closed paths and their zig-zag walks}
In according to \cite{Cisto_Navarra_closed_path}, we recall the definition of a \textit{closed path polyomino}, and the configuration of cells characterizing its primality. Closed paths are polyominoes contained in the class of \textit{thin polyominoes}, introduced in \cite{MRR22primalityOfPolyominoes}, that are polyominoes not containing the square tetromino as subpolyomino. 

\begin{definition}
We say that a polyomino $\cP$ is a \textit{closed path} if it is a sequence of cells $A_1,\dots,A_n, A_{n+1}$, $n>5$, such that:
\begin{enumerate}
	\item $A_1=A_{n+1}$;
	\item $A_i\cap A_{i+1}$ is a common edge, for all $i=1,\dots,n$;
	\item $A_i\neq A_j$, for all $i\neq j$ and $i,j\in \{1,\dots,n\}$;
	\item For all $i\in\{1,\dots,n\}$ and for all $j\notin\{i-2,i-1,i,i+1,i+2\}$ then $V(A_i)\cap V(A_j)=\emptyset$, where $A_{-1}=A_{n-1}$, $A_0=A_n$, $A_{n+1}=A_1$ and $A_{n+2}=A_2$. 
\end{enumerate}
\end{definition}

\begin{figure}[h!]
	\centering
	\includegraphics[scale=0.6]{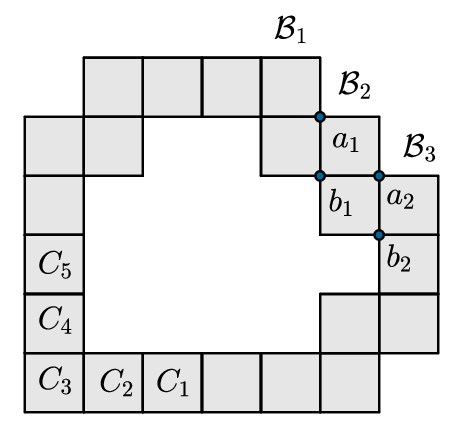}
	\caption{A closed path with an $L$-configuration and a ladder of three steps.}
	\label{Figura:L conf + Ladder}
\end{figure}
A path of five cells $C_1, C_2, C_3, C_4, C_5$ of $\cP$ is called an \textit{L-configuration} if the two sequences $C_1, C_2, C_3$ and $C_3, C_4, C_5$ go in two orthogonal directions. A set $\cB=\{\cB_i\}_{i=1,\dots,n}$ of maximal horizontal (or vertical) blocks of rank at least two, with $V(\cB_i)\cap V(\cB_{i+1})=\{a_i,b_i\}$ and $a_i\neq b_i$ for all $i=1,\dots,n-1$, is called a \textit{ladder of $n$ steps} if $[a_i,b_i]$ is not on the same edge interval of $[a_{i+1},b_{i+1}]$ for all $i=1,\dots,n-2$. For instance, in Figure \ref{Figura:L conf + Ladder} there is a closed path having an $L$-configuration and a ladder of three steps. A closed path has no zig-zag walks if and only if it contains an $L$-configuration or a ladder of at least three steps (see \cite[Section 6]{Cisto_Navarra_closed_path}).		

\begin{definition}\label{def:zig-zagwalk}
    A {\em zig-zag walk} of a collection of cells $\MP$ is a sequence $\MW : I_1,\ldots,I_l$ of distinct inner intervals of $\MP$ where,
    for all $i= 1,\ldots,l$, the interval $I_i$ has diagonal corners $v_i,z_i$ and anti-diagonal corners $u_i, v_{i+1}$
    or anti-diagonal corners $v_i,z_i$ and diagonal corners $u_i, v_{i+1}$, such that
    \begin{enumerate}
        \item \label{def:zig-zagwalk:intersection}
            $I_1 \cap I_l = \{v_1 = v_{l+1}\}$ and $I_i\cap I_{i+1} = \{v_{i+1}\}$ for all $i=1,\ldots,l-1$;
        \item \label{def:zig-zagwalk:edgeinterval}
            $v_i$ and $v_{i+1}$ are on the same edge interval of $\MP$, for all $i=1,\ldots, l$;
        \item \label{def:zig-zagwalk:ziintersection}
            for all $i,j\in \{1,\ldots,l\}$ with $i\neq j$, there exists no interval $J$ of $\MP$ such that $z_i$ and $z_j$ belongs to $J$. 
    \end{enumerate}
\end{definition}

\noindent By using the concept of zig-zag walk, in \cite[Theorem 6.2]{Cisto_Navarra_closed_path} the authors characterize the primality of a closed path $\cP$ proving that $I_{\cP}$ is prime if and only if $\cP$ does not contain any zig-zag walk. \\
Let $\cP$ be a collection of cells. Following the notations in \cite{MRR20primalityOfPolyominoes}, for a zig-zag walk $\MW : I_1,\ldots,I_l$ of $\MP$, let $f_\MW$ be the corresponding binomial of $\MW$, that is $f_\cW=\prod_{i=1}^{l}x_{z_i} - \prod_{i=1}^{l}x_{u_i}$. 
Moreover, we denote 
$$N(\cW)=\{v \in V(\MP) : v\in [v_i, v_{i+1}]\ \text{for} \ 1\leq i \leq l\}$$ and we call it the \emph{necklace} of $\cW$. Furthermore, it is shown in \cite[Proposition\ 3.5]{MRR20primalityOfPolyominoes} that if a polyomino $\cP$ has a zig-zag walk $\cW$, then:
$$\pm x_{v_1}f_\cW=\prod_{k>1}x_{z_k}f_{I_1}+\ldots+(-1)^{i+1}\prod_{j<i}x_{u_j}\prod_{k>i}x_{z_k}f_{I_i}+\ldots+(-1)^{l+1}\prod_{j<l}x_{u_j}f_{I_l}$$
where $f_{I_i}$ is the binomial related to the inner interval $I_i$, for $i\in\{1,\ldots,l\}$ and the sign depends if $v_1$ is a diagonal or an anti-diagonal corner of $I_1$. For $i\in \{2,\ldots,l\}$, it is easy to obtain a similar expression also for $x_{v_i}f_\cW$, by relabelling the indices $1,2,\ldots,l$. In particular, for all zig-zag walk $\cW$ then $x_{v_i}f_\cW\in I_\cP$ for all $1\leq i\leq l$. Observing the previous formula, we can also extend the same claim also in the case $\cP$ is a collection of cells. 

\begin{Lemma}\label{lem:onezigzagwalk}
    Let $\cP$ be a collection of cells with a zig-zag walk $\MW$ and $\pp$ be a minimal prime ideal of $I_\cP$. 
Under the notations as above, if $f_\MW \notin \pp$, then $x_v\in \pp$ for all $v\in N(\MW)$.
\end{Lemma}
\begin{proof}
    Since $f_\MW\notin \pp$, by above discussion, we get that $x_{v_i}\in \pp$ for all $1\leq i\leq l$.
Fix an $i$. We have $x_{u_i}x_{v} - x_{v_i}x_{b_v} \in I_{\MP}$ for all $v\in ]v_i, v_{i+1}[$, where $b_v\in V(\MP)$ such that the rectangle with the vertices $u_i, v, v_i$ and $b_v$ is in $\MP$.
    Since $x_{v_i} \in \pp$, $x_vx_{u_i}\in \pp$. 
    Therefore, either $x_v\in \pp$ for all $v\in ]v_i, v_{i+1}[$ or $x_{u_i}\in \pp$.\\
    \noindent Similarly, by repeating the above argument for $x_{b_v}x_{v_{i+1}} - x_vx_{z_i} \in I_{\MP}$ for all $v\in ]v_i, v_{i+1}[$, where $b_v\in V(\MP)$ such that the rectangle with the vertices $z_i, v, v_{i+1}$ and $b_v$ is in $\MP$. 
We get that either $x_v\in \pp$ for all $v\in ]v_i, v_{i+1}[$ or $x_{z_{i}}\in \pp$.
If $x_{u_i}, x_{z_i}\in \pp$, then  $f_\MW \in \pp$ which contradicts the assumption. Hence, $x_v \in \pp$ for all $v\in ]v_i, v_{i+1}[$. 
\end{proof}

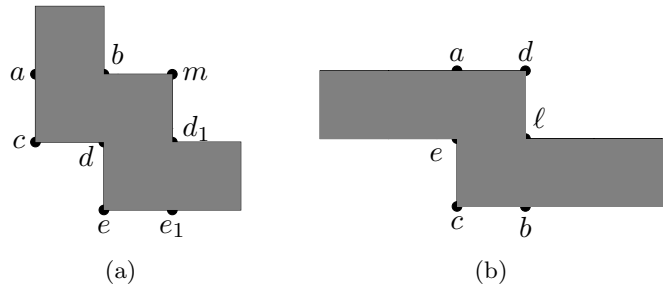
\begin{figure}[h]
	
	\subfloat[]{
		\begin{tikzpicture}[scale=0.9]
			\draw[] (0,2)--(2,2)--(2,0)--(1,0)--(1,1)--(0,1)--(0,2) 
			(1,2)--(1,1)--(2,1)
			(2,0)-- (3,0)--(3,1)--(2,1)
			(0,2)--(0,3)--(1,3)--(1,2);
			\filldraw[black] (2,2) circle (2.0pt) node[anchor=west]  {$m$};
			\filldraw[black] (1,0) circle (2.0pt) node[anchor=north]  {$e$};
			\filldraw[black] (2,0) circle (2.0pt) node[anchor=north]  {$e_1$};
			\filldraw[black] (1,.8) circle (0pt) node[anchor=east]  {$d$};
			\filldraw[black] (1,1) circle (2.0pt);
			\filldraw[black] (0,1) circle (2.0pt) node[anchor=east]  {$c$};
			\filldraw[black] (1.2,2) circle (0pt) node[anchor=south]  {$b$};
			\filldraw[black] (1,2) circle (2.0pt);
			\filldraw[black] (2,1.2) circle (0pt) node[anchor=west]  {$d_1$};
			\filldraw[black] (2,1) circle (2.0pt);
			\filldraw[black] (0,2) circle (2.0pt) node[anchor=east]  {$a$};
			\filldraw[black] (1.5,.5) circle (0pt) node[anchor=center]  {$Z$};
			\filldraw[black] (1.5,1.5) circle (0pt) node[anchor=center]  {$X$};
			\filldraw[black] (.5,1.5) circle (0pt) node[anchor=center]  {$Y$};
			\filldraw[black] (2.5,.5) circle (0pt) node[anchor=center]  {$W$};
			
			\fill[fill=gray, fill opacity=0.3] (0,1)-- (0,3)--(1,3)--(1,2)--(2,2)--(2,1)--(3,1)--(3,0)--(1,0)--(1,1);
		\end{tikzpicture}
	\label{fig:threestepladder}
	}\qquad
	\subfloat[]{
		\begin{tikzpicture}[scale=0.9]
			\draw[] (0,2)--(2,2)--(2,0)--(1,0)--(1,1)--(0,1)--(0,2) 
			(1,2)--(1,1)--(2,1)
			(2,0)-- (4,0)--(4,1)--(2,1) (3,0)--(3,1)
			(0,2)--(-1,2)--(-1,1)--(0,1);
			\filldraw[black] (2,2) circle (2.0pt) node[anchor=south]  {$d$};
			\filldraw[black] (1,0) circle (2.0pt) node[anchor=north]  {$c$};
			\filldraw[black] (2,0) circle (2.0pt) node[anchor=north]  {$b$};
			\filldraw[black] (1,2) circle (2.0pt) node[anchor=south]  {$a$};
			\filldraw[black] (.7,1) circle (0pt) node[anchor=north]  {$e$};
			\filldraw[black] (1,1) circle (2.0pt);
			\filldraw[black] (2.2,1) circle (0pt) node[anchor=south]  {$\ell$};
			\filldraw[black] (2,1) circle (2.0pt);
			\filldraw[black] (.5,1.5) circle (0pt) node[anchor=center]  {$W$};
			\filldraw[black] (1.5,1.5) circle (0pt) node[anchor=center]  {$X$};
			\filldraw[black] (1.5,.5) circle (0pt) node[anchor=center]  {$Y$};
			\filldraw[black] (2.5,.5) circle (0pt) node[anchor=center]  {$Z$};
			\fill[fill=gray, fill opacity=0.3] (1,0)-- (4,0)--(4,1)--(1,1);
			\fill[fill=gray, fill opacity=0.3] (2,2)-- (-1,2)--(-1,1)--(2,1);
	\end{tikzpicture}
	\label{fig:Lshape}}
	\caption{Possible changes of direction in a closed path having a zig-zag walk.}
\end{figure}

\noindent From now onwards, we assume that $\MP$ is a non-prime closed path polyomino. 
By~\cite[Theorem\ 6.2]{Cisto_Navarra_closed_path}, $\MP$ contains a zig-zag walk. As a consequence, by \cite[Proposition 6.1]{Cisto_Navarra_closed_path} $\cP$ does not contain any $L$-configuration, so any two maximal inner intervals of $\MP$ intersect themselves in the cells  displayed in Figure~\ref{fig:threestepladder} or~\ref{fig:Lshape} (up to reflections or rotations).  
We wish to determine all zig-zag walks of $\MP$. 

\begin{discussionbox}\label{dis:zig-zag}
Let $\MW : I_1,\ldots,I_l$ be a zig-zag walk of $\MP$. 
We show that from the labelled vertices of Figure~\ref{fig:threestepladder}, only $d\in \{v_0,\ldots,v_{l}\}$.
By~\eqref{def:zig-zagwalk:ziintersection} of Definition~\ref{def:zig-zagwalk}, $I_i$ and $I_{j}$ are not contained in a maximal inner interval if $i\neq j$.
By~\eqref{def:zig-zagwalk:intersection} of Definition~\ref{def:zig-zagwalk}, $m, a, c, e$ and  $e_1$ as in Figure~\ref{fig:threestepladder} do not belong to $\{v_0,\ldots,v_{l}\}$.  
Since $[v_1,v_2]\cup \cdots \cup [v_l,v_1]$ is a closed path of edge intervals, we get that $d\in \{v_0,\ldots,v_{l}\}$.
Now it suffice to show that $d_1, b \notin \{v_0,\ldots,v_{l}\}$.
Assume that $d_1 \in \{v_0,\ldots,v_{l}\}$.
Then, there exists a $j$ such that $d, d_1 \in V(I_j)$. 
By Definition~\ref{def:zig-zagwalk}, $I_j =\{X\}$ or $I_j=\{Z\}$.
We may assume that $I_{j-1}$ is the inner interval such that $d_1\in V(I_{j-1})$ and $W \in I_{j-1}$, and $I_{j+1}$ is the inner interval such that $d\in V(I_{j+1})$ and $Y \in I_{j+1}$.
If $I_j =\{X\}$, then $b, d\in I_j \cap I_{j+1}$ and if $I_j =\{Z\}$, then $d_1, e_1 \in I_j \cap I_{j-1}$.
Both of them contradict~\eqref{def:zig-zagwalk:intersection} of Definition~\ref{def:zig-zagwalk};
thus, $d_1 \notin \{v_0,\ldots,v_{l}\}$. 
Similarly,  $b \notin \{v_0,\ldots,v_{l}\}$. 
Hence, from the labelled vertices of Figure~\ref{fig:threestepladder}, only $d\in \{v_0,\ldots,v_{l}\}$.\\
\noindent We now show that from the labelled vertices of Figure~\ref{fig:Lshape}, exactly one of $e$ and  $\ell$ is in $\{v_0,\ldots,v_{l}\}$.
By~\eqref{def:zig-zagwalk:intersection} of Definition~\ref{def:zig-zagwalk}, $a, b, c$ and  $d$ as in Figure~\ref{fig:Lshape} do not belong to $\{v_0,\ldots,v_{l}\}$.  
Since $v_i$ is the vertex where the inner intervals $I_{i-1}$ and $I_{i}$ intersects and, $e$ and $\ell$ are the vertices belongs to two maximal inner intervals, either $e\in \{v_0,\ldots,v_{l}\}$ or $\ell \in \{v_0,\ldots,v_{l}\}$. 
Now, we show that exactly one of $e$ and $\ell$ is in $\{v_0,\ldots,v_{l}\}$.
Suppose both $e$ and $\ell$ are in $\{v_0,\ldots,v_{l}\}$. 
Then there exists a $j\in \{1,\ldots,l\}$ such that $I_j = \{X\}$ or $I_j=\{Y\}$.
Let $I_{j-1}$ (respectively $I_{j+1}$) be the inner interval such that $e\in V(I_{j-1})$ (respectively $\ell\in V(I_{j+1})$) and $W \in I_{j-1}$ (respectively $Z\in I_{j+1}$).
If $I_j = \{X\}$, then $a, e \in I_{j-1} \cap I_j$ which is a contradiction.
Similarly, if $I_j = \{Y\}$, then $b, \ell\in I_{j} \cap I_{j+1}$ which is a contradiction. Hence, exactly one of $e$ and $\ell$ is in $\{v_0,\ldots,v_{l}\}$.
\end{discussionbox}

\begin{figure}[h]
\begin{center}
\begin{tikzpicture}[scale=0.8]
    \draw[] (1,0)--(8,0)--(8,1)--(1,1)--(1,1)--(1,0) 
    (2,0)--(2,2) (3,0)--(3,1) (4,0)--(4,1) (5,0)--(5,1) (6,0)--(6,1) (6,0)--(6,1) (7,0)--(7,1)
    (0,2)--(2,2) (0,1)--(1,1)--(1,2)
    (0,1)--(0,4)--(1,4)--(1,2)--(0,2) (0,3)--(2,3)--(2,5)
    (1,4)--(4,4)--(4,5)--(1,5)--(1,4) (3,4)--(3,5)
    (3,4)--(3,3)--(6,3)--(6,4)--(4,4)--(4,3) (5,4)--(5,3)
    (5,4)--(5,5)--(8,5)--(8,4)--(6,4)--(6,5) (7,4)--(7,5)
    (7,4)--(7,3)--(8,3)--(8,4)
    (8,4)--(9,4)--(9,1)--(8,1)--(8,3)--(9,3)
    (9,2)--(7,2)--(7,1)
    ;

    \filldraw[black] (.7,1) circle (0pt) node[anchor=north]  {$d_{1}$};
    \filldraw[black] (1,1) circle (2.0pt);
    \filldraw[black] (.7,4) circle (0pt) node[anchor=south]  {$d_{2}$};
    \filldraw[black] (1,4) circle (2.0pt);
    \filldraw[black] (8.3,1) circle (0pt) node[anchor=north]  {$d_{4}$};
    \filldraw[black] (8,1) circle (2.0pt);
    \filldraw[black] (8.3,4) circle (0pt) node[anchor=south]  {$d_{3}$};
    \filldraw[black] (8,4) circle (2.0pt);
   
    \filldraw[black] (2.7,4) circle (0pt) node[anchor=north]  {$b_{1}$};
    \filldraw[black] (3,4) circle (2.0pt);
    \filldraw[black] (3.7,4) circle (0pt) node[anchor=north]  {$b_{2}$};
    \filldraw[black] (4,4) circle (2.0pt);
    \filldraw[black] (4.7,4) circle (0pt) node[anchor=south]  {$c_{1}$};
    \filldraw[black] (5,4) circle (2.0pt);
    \filldraw[black] (5.7,4) circle (0pt) node[anchor=south]  {$c_{2}$};
    \filldraw[black] (6,4) circle (2.0pt);
    
    \fill[fill=gray, fill opacity=0.3] (1,0) -- (8,0)-- (8,1) -- (1,1);
    \fill[fill=gray, fill opacity=0.3] (2,1) -- (2,2)-- (0,2) -- (0,1);
    \fill[fill=gray, fill opacity=0.3] (0,2) -- (1,2)-- (1,4) -- (0,4);
    \fill[fill=gray, fill opacity=0.3] (1,3) -- (2,3)-- (2,5) -- (1,5);
    \fill[fill=gray, fill opacity=0.3] (2,4) -- (4,4)-- (4,5) -- (2,5);

    \fill[fill=gray, fill opacity=0.3] (3,3) -- (6,3)-- (6,4) -- (3,4);
    \fill[fill=gray, fill opacity=0.3] (5,4) -- (8,4)-- (8,5) -- (5,5);
    \fill[fill=gray, fill opacity=0.3] (7,3) -- (9,3)-- (9,4) -- (7,4);
    \fill[fill=gray, fill opacity=0.3] (9,3) -- (9,1)-- (7,1) -- (7,2)--(8,2)--(8,3);
\end{tikzpicture}
\caption{An example of a closed path having zig-zag walks.}\label{fig:examplezig-zag}
\end{center}
\end{figure}
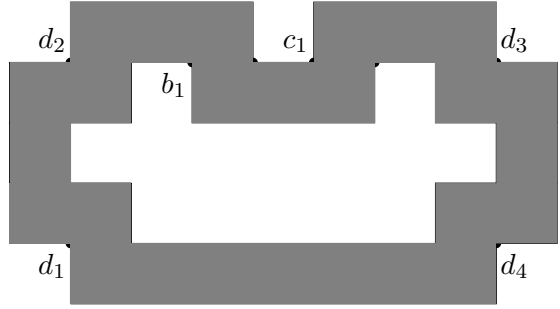

\begin{example}
    We illustrate the above discussion now.
    Let $\MP$ be the polyomino as shown in the Figure~\ref{fig:examplezig-zag}.
    Note that $\MP$ is a non-prime polyomino.
    Let $\MW : I_1,\ldots,I_l$ be a zig-zag walk of $\MP$.
    By Discussion~\ref{dis:zig-zag}, $\{d_1,d_2, d_3, d_4\} \subset \{v_1,\ldots, v_l\}$. 
    Also, by Discussion~\ref{dis:zig-zag}, exactly one of $b_1$ and $b_2$ is in $\{v_1,\ldots, v_l\}$ and exactly one of $c_1$ and $c_2$ is in $\{v_1,\ldots, v_l\}$.
    So in conclusion, we get that $l=6$ and $\{v_1,\ldots, v_6\}$ is either $\{d_1,d_2, d_3, d_4, b_1, c_1\}$ or 
   $\{d_1,d_2, d_3, d_4, b_1, c_2\}$ or $\{d_1,d_2, d_3, d_4, b_2, c_1\}$ or $\{d_1,d_2, d_3, d_4, b_2, c_2\}$.
\end{example}

\begin{remark}
    Based on the Discussion~\ref{dis:zig-zag}, we make few remarks on the zig-zag walks of a non-prime closed path polyomino.
    \begin{enumerate}
        \item 
            The number of inner intervals in a zig-zag is equal to the number of maximal inner intervals in $\MP$ with at least three cells.
        \item
            Let $\MW : I_1,\ldots,I_l$ and $\MW': I'_1, \ldots, I'_l$ be two zig-zag walks of $\MP$. Then $N(\cW)=N(\cW')$.
    \item\label{rem:zigzagbinomial} Under the labelling of Figure~\ref{fig:threestepladder} $x_b,x_{d_1}\in N(\MP)$ and if $f_\cW=\prod_{i=1}^{l}x_{z_i} - \prod_{i=1}^{l}x_{u_i}$ is any zig-zag walk then either $x_c\in \mathrm{supp}(\prod_{i=1}^{l}x_{z_i})$ and $x_e\in \mathrm{supp}(\prod_{i=1}^{l}x_{u_i})$ or $x_c\in \mathrm{supp}(\prod_{i=1}^{l}x_{u_i})$ and $x_e\in \mathrm{supp}(\prod_{i=1}^{l}x_{z_i})$.        
    \end{enumerate}
    \label{necklace}
\end{remark}

\begin{notation}\label{not:primeideals}
\rm    Consider the following notation:
    \begin{itemize}
    \item $N(\MP)\eqdef N(\cW)$, for some zig-zag walk $\cW$ of $\cP$. We observed that this set is unique and does not depend on the choice of the zig-zag walk $\cW$.
        \item $M(\cP)\eqdef\{m \in V(\cP) \ \text{as shown in Figure~\ref{fig:threestepladder}, up to reflections or rotations}\}$
        \item $R(\cP)\eqdef\{x_ax_b-x_cx_d: a, b,c, d \ \text{are as shown in Figure~\ref{fig:Lshape}, up to reflections or rotations}\}$
    \end{itemize}

Define the following ideals:
\[\pp_1 = I_\MP + (f_\MW : \MW \ \text{is a zig-zag walk of}\ \MP)\]
   \[\pp_2 = (x_a : a\in N(\MP)) +(x_m : m \in M(\MP))+(R(\cP)).\]
\end{notation}

\begin{Lemma}\label{lem:i_pinp2} 
    Let $\cP$ be a non-prime closed path polyomino. Then $I_\MP \subset \pp_2$.
\end{Lemma}

\begin{proof}
	Let $[p,q]$ be an inner interval of $\cP$, with $r,s$ as anti-diagonal corners. We prove that $f=x_px_q-x_rx_s$ belongs to $\pp_2$. Assume that, referring to Figure \ref{fig:Lshape}, $q\in\ \{a,d,\ell\}$ or $p\in\{e,c,b\}$. If $f\in R(\cP)$ then $f\in \pp_2$ trivially. Suppose that $f \notin R(\cP)$ and $q=d$. Then $s=\ell$, otherwise $f\in R(\cP)$, and in particular $p$ belongs to the edge interval of $\cP$ containing $e$ and $\ell$. It follows that $p,s\in N(\cP)$, so $x_px_q,x_rx_s\in \pp_2$, hence $f\in \pp_2$. The other cases can be proved in a similar way. Now assume that, referring to Figure \ref{fig:threestepladder}, $q\in \{b,m,d_1\}$ or $p\in\{a,c,d,e,e_1\}$. We may suppose that $q=m$, because one can prove all other cases with similar arguments. In such a case $p\in\{c,d,e_1\}$. Suppose that $p=c$. Then $r=a$ and $s=d_1$. Since $q=m\in M(\cP)$ and $d_1\in N(\cP)$, we get $x_px_q,x_rx_s\in \pp_2$, hence $f\in \pp_2$. The other two cases $p=d$ and $p=e$ can be proved similarly. For the other situations, coming from reflections or rotations of Figures \ref{fig:threestepladder} and \ref{fig:Lshape}, the proofs are similar to the previous ones. The last case to discuss is when no corners of $[p,q]$ is one of the vertices in Figures \ref{fig:threestepladder} and \ref{fig:Lshape}. In such a case, we can have one of the following four situations: either $r,q\in N(\cP)$, or $p,s\in N(\cP)$, or $p,r\in N(\cP)$, or $q,s\in N(\cP)$, depending on the shape of $\cP$. In each case it follows that $x_px_q,x_rx_s\in \pp_2$, hence $f\in \pp_2$.  
\end{proof}
   
\begin{proposition}\label{prop:p_2prime}
    Let $\cP$ be a non-prime closed path polyomino. Then, under the notations of Notation~\ref{not:primeideals}, $\pp_2$ is a prime ideal and $\mathrm{height}(\pp_2)=|\cP|$.
\end{proposition}
   
\begin{proof}
    The minimal generating set of $\pp_2$ can be partitioned between disjoint set of variables. 
    So that $\pp_2$ admits a tensor product decomposition, where one ideal is generated by variables and others are principal prime ideals.
    Hence, $\pp_2$ itself is a prime ideal.

    Under the labelling of Figure~\ref{fig:threestepladder}, the variables $x_m$ corresponds to the cell $X$. 
    \noindent Under the labelling of Figure~\ref{fig:Lshape}, the generator $x_ax_b-x_cx_d$ corresponds to the cell $X$.
    So it suffices to show that the number of remaining cells is equal to the cardinality of $N(\MP)$.
    Let $\MW : I_1,\ldots,I_l$ be a zig-zag walk of $\MP$. 
    For $1\leq i \leq l$, the number of cells in $I_i$ are equal to the number of vertices in $]v_i,v_{i+1}]$. 
    Therefore, the  number of remaining cells is equal to the cardinality of $N(P)$.
    Thus, height of $\pp_2$ is  equal to number of cells of $\MP$
\end{proof}


In the following proposition, we show that if a polyomino $\MP$ becomes simple after removing a cell, then $\mathrm{height}(I_\mathcal{P}) = |\mathcal{P}|$. This gives a different proof of~\cite[Corollary 3.5]{konig_type}. 
As a consequence the ideal $\mathfrak{p}_2$ is a minimal prime of $I_\mathcal{P}$. 


\begin{proposition}\label{prop:I_Pheight}
Let $\cP$ be a polyomino. If there exists a cell $A$ in $\MP$ such that $\cP\setminus \{A\}$ is simple, then $\mathrm{height}(I_\cP)=|\cP|$.
In particular, it holds for closed path polyominoes.
\label{height-P}
\end{proposition}
\begin{proof}
    It is known by \cite[Theorem 3.1]{height} that $\mathrm{height}(I_\cP)\leq |\cP|$ for every collection of cells $\cP$. Let $\cP'=\cP \setminus \{A\}$. Since $\cP'$ is simple, $I_{\cP'}$ is a prime ideal (see \cite{QSS}). 
    Also, by \cite[Theorem\ 2.1]{coordinate} and \cite[Corollary 2.3]{balanced}, we have that $\mathrm{height}(I_{\cP'})=|\cP'|=|\cP|-1$. We show that $\mathrm{height}(I_\mathcal{P'}) < \mathrm{height}(I_\mathcal{P})$. We have that $I_\mathcal{P'}\subsetneq I_\mathcal{P}$ and $I_\mathcal{P'}$ is prime. Let $\mathfrak{p}$ be a prime ideal such that $I_\mathcal{P}\subseteq \mathfrak{p}$. In particular $I_\mathcal{P'}\subsetneq \mathfrak{p}$. Let $\mathfrak{p}_0 \subseteq \cdots \subseteq \mathfrak{p}_n=I_\mathcal{P'}$ be a chain of prime ideals contained in $I_\mathcal{P'}$ of length $n=\mathrm{height}(I_\mathcal{P'})$. Then $\mathfrak{p}_0 \subseteq \cdots \subseteq \mathfrak{p}_n=I_\mathcal{P'}\subsetneq \mathfrak{p}$ is a chain of prime ideals contained in $\mathfrak{p}$, in particular we have $\mathrm{height}(\mathfrak{p})\geq n+1$. Since $\mathrm{height}(I_\mathcal{P}):=\min \{\mathrm{height}(\mathfrak{q})\mid \mathfrak{q}\supseteq I_\mathcal{P}\ \mbox{and}\ \mathfrak{q}\ \mbox{is prime}\}$, we obtain $\mathrm{height}(I_\mathcal{P})\geq n+1$, so $\mathrm{height}(I_\mathcal{P'}) < \mathrm{height}(I_\mathcal{P})$ as claimed. Therefore $|\cP|-1=\mathrm{height}(I_{\cP'})<\mathrm{height}(I_\cP)\leq |\cP|$. So the only possibility is $\mathrm{height}(I_\cP)=|\cP|$.

\noindent For the particular case, from \cite[Lemma 3.3]{Cisto_Navarra_closed_path} we know that every closed path polyomino $\cP$ contains a block of rank at least three, so there exists a cell $A$ in $\cP$ such that $\cP \setminus \{A\}$ is simple.
\end{proof}
   
\subsection{Primality of $\mathfrak{p}_1$ and its height}

	
%

\noindent Let $\cP$ be a polyomino. A binomial $f=f^+-f^-$ in a binomial ideal $J\subset S_\cP$ is called \textit{redundant} if it can be expressed as a linear combination of binomials in $J$ of lower degree. A binomial is called \textit{irredundant} if it is not redundant. Moreover, we denote by $V^{+}_f$ the set of the vertices $v$, such that $x_v$ appears in $f^+$, and by $V^{-}_f$ the set of the vertices $v$, such that $x_v$ appears in $f^-$.

	\begin{Lemma}
	Let $\cP$ be a polyomino and $\phi: S_\cP\rightarrow T$ a ring homomorphism with $T$ an integral domain. Let $J=\ker \phi$ and $f=f^+-f^-$ be a binomial in $J$ with $\deg f\geq3$. Suppose that $I_\cP \subseteq J$ and $\phi(x_v)\neq 0$ for all $v\in V(\cP)$. 
	 \begin{enumerate}
	\item If there exist three vertices $p,q\in V^+_{f}$ and $r\in V^-_{f}$ such that $p,q$ are diagonal (respectively  anti-diagonal) corners of an inner interval and $r$ is one of the anti-diagonal (respectively diagonal) corners of the same inner interval, then $f$ is redundant in $J$. The same claim also holds if $p,q\in V^-_f$ and $r\in V^+_f$.
	\item Suppose that $\cP$ has a subpolyomino having a collection of cells as in Figure~\ref{3points}, and that $a,e_1\in V_f^+$ (respectively $V_f^-$), $d\in V_f^-$ (respectively $V_f^+$) and $e_i,m\in V_f^-$ (respectively $V_f^+$) for some $i\in \{2,\ldots,n\}$. Then $f$ is redundant. 
	\begin{figure}[h!]
		\centering
		\includegraphics[scale=0.7]{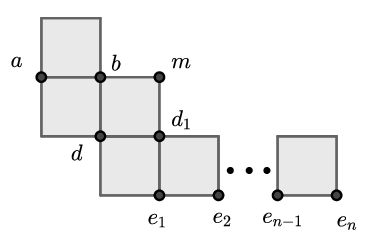}
		\caption{Configuration examined in Lemma~\ref{Lemma shikama1}.}
		\label{3points}
	\end{figure}
	 \end{enumerate}
	 \label{Lemma shikama1}
    \end{Lemma}
        
\begin{proof}
 For the proof of $(1)$ see \cite[Lemma 2.2]{Cisto_Navarra_closed_path}. We prove $(2)$. 
Observe that: $f=f^+ -f^-=f^+  -\frac{f^-}{x_d x_m}x_b x_{d_1} - \frac{f^-}{x_d x_m}(x_d x_m-x_b x_{d_1})$ and $\frac{f^-}{x_d x_m}(x_d x_m-x_b x_{d_1})\in I_\cP\subseteq J$. Denote $\tilde{f}=f^+  -\frac{f^-}{x_d x_m}x_b x_{d_1}$ and observe that $\tilde{f}\in \ker \phi=J$. Moreover $e_1 \in V_f^+$ and $x_{d_1},x_{e_i}\in V_{\tilde{f}}^-$, since $\tilde{f}^-= \frac{f^-}{x_d x_m}x_b x_{d_1}$. So $\tilde{f}$ is redundant in $J$ by Lemma~\ref{Lemma shikama1}(1) and in particular also $f$ is redundant in $J$. The same claim holds if we suppose $a,e_1\in V_f^-$, $d\in V_f^+$ and $e_i\in V_f^+$, using the same argument.
\end{proof} 

\noindent Let $\cP$ be a closed path having a zig-zag walk and consider a sub-polyomino of $\cP$ having the shape as in Figure~\ref{fig:threestepladder}, up to reflections and rotations.
	\noindent Let $\{V_i\}_{i\in I}$ be the sets of the maximal vertical edge intervals of $\cP$ and $\{H_j\}_{j\in J}$ be the set of the maximal horizontal edge intervals of $\cP$. Let $\{v_i\}_{i\in I}$ and $\{h_j\}_{j\in J}$ be the set of the variables associated respectively to $\{V_i\}_{i\in I}$ and $\{H_j\}_{j\in J}$. Let $w$ be another variable different from $v_i$ and $h_j$, for $i\in I$ and $j\in J$.  
	Under the notation of Figure \ref{fig:threestepladder}, we define the following map:
\begin{align*}
	\alpha: V(\cP)&\longrightarrow K[\{v_i,h_j,w\}:i\in I,j\in J]\\
	r&\longmapsto  v_ih_jw^k
	\end{align*}
with $r\in V_i\cap H_j$, $k=0$ if $r\notin \{a,b,c,d,e\}$, and $k=1$, if $r\in \{a,b,c,d,e\}$. \\
Denote $T_{\cP}:=K[\alpha(r):r\in V(\cP)]$, that is a toric ring. We consider the following surjective ring homomorphism
	\begin{align*}
	\psi: S_\cP &\longrightarrow T_{\cP}\\
	\psi(x_r&)=\alpha(r)
	\end{align*}
    Let us denote the kernel of $\psi$ by $J_{\cP}$. Observe that the ideal $J_\cP$ can be viewed as the toric ideal defined in \cite[Section 3]{MRR20primalityOfPolyominoes}, in particular we can deduce that $I_\cP=(J_\cP)_2$.

\begin{remark} \rm
	
Let $f$ be an irredundant binomial in $J_\cP$ and $a\in V_f^+$ (resp. $V_f^-$). Denote by $H_a$ and $V_a$ respectively the maximal horizontal and vertical edge intervals of $\cP$ containing $a$. Then there exist $a_1\in H_a$ and $a_2\in V_a$ such that $a_1,a_2\in V_f^-$ (resp. $V_f^+$). Moreover if $\mathcal{I}$ is a maximal edge interval of $\cP$, then $|V_f^+ \cup \mathcal{I}|=|V_f^- \cup \mathcal{I}|$ otherwise $f\notin \ker \psi$, and $V_f^+ \cap V_f^-=\emptyset$ otherwise $f$ is redundant.
\end{remark}

\begin{Lemma}
Let $\cP$ be a closed path polyomino and $f\in J_P$ an irredundant binomial with $\deg f \geq 3$. Consider a subpolyomino having a configuration as in Figure~\ref{Fig:Lemma1}.

\begin{figure}[h!]
\centering
\includegraphics[scale=0.8]{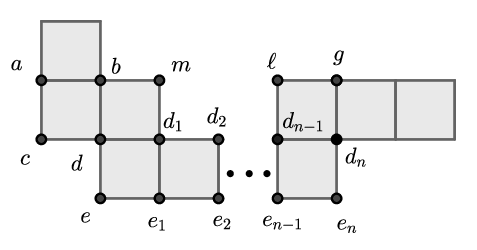}
\caption{Configuration examined in Lemma~\ref{Lemma1}.}
\label{Fig:Lemma1}
\end{figure}

\noindent Suppose that $c\in V_f^{+}, e\in V_{f}^-$. Then the following hold:
\begin{enumerate}
\item $e_i,d_i \notin V_{f}^+\cup V_{f}^-$ for all $i\in \{1,\ldots,n-2\}$.
\item Exactly one of the following occurs:
\begin{itemize}
\item[a)] $e_{n-1}\in V_{f}^+$, $\ell \in V_{f}^-$ and $e_n \notin V_{f}^+\cup V_{f}^-$.
\item[b)] $e_{n}\in V_{f}^+$, $g \in V_{f}^-$ and $e_{n-1} \notin V_{f}^+\cup V_{f}^-$.
\end{itemize}
\item $d_{n-1},d_{n} \notin V_{f}^-$.
\end{enumerate} 
The same result holds replacing $V_{f}^+$ with $V_{f}^-$ and $V_{f}^-$ with $V_{f}^+$ in all previous occurrences.
\label{Lemma1}
\end{Lemma}

\begin{proof}
 $(1).$ Firstly, for $i\in \{2,\ldots,n-2\}$ observe that $e_i\in V_f^+$ if and only if $d_i\in V_f^-$. In particular, if $e_i\in V_f^+$ and $d_i\in V_f^-$ we obtain a contradiction with Lemma~\ref{Lemma shikama1}(1) considering $e, e_i, d_i$. So, $e_i\notin V_f^+$ and $d_i\notin V_f^-$ for all $i\in \{2,\ldots,n-2\}$. Moreover, if $e_1\in V_f^+$ then, considering its maximal vertical edge interval, we have $d_1\in V_f^-$ or $m\in V_f^-$, obtaining in each case a contradiction with Lemma~\ref{Lemma shikama1}(1) considering respectively $e,e_1,d_1$ and $e,e_1,m$. So $e_1\notin V_f^+$. Since $e \in V_f^-$, in its maximal horizontal edge interval there exists a vertex $X\in V_f^+$, and by the above arguments the only possibilities is $X=e_{n-1}$ or $X=e_n$. In both cases observe that, for $i\in \{2,\ldots,n-1\}$, $e_i\in V_f^-$ if and only if $d_i\in V_f^+$, obtaining that $e_i\notin V_f^-$ and $d_i\notin V_f^+$ for all $i\in \{3,\ldots,n-1\}$, otherwise we have a contradiction with Lemma~\ref{Lemma shikama1}(1) considering $X, e_i, d_i$. Suppose $e_1\in V_f^-$. Then we have necessarily $e_{n-1},e_n\in V_f^+$ and, considering its vertical edge interval, the following two possibilities occur: either $d_1\in V_f^+$ or $m\in V_f^+$. If $d_1\in V_f^+$, then we have a contradiction by Lemma \ref{Lemma shikama1}(1) considering $e_1,d_1,e_{n-1}$. If $m\in V_f^+$, then $a\in V_f^-$ or $b\in V_f^-$. Both cases lead to a contradiction by Lemma \ref{Lemma shikama1}(1) considering respectively $c,m,a$ and $b,m,e_1$. For what concern the vertex $d_1$, we obtain $d_1\notin V_f^+ \cap V_f^-$ by similar arguments. \\
 $(2)-(3).$ We showed before that either $e_{n-1}\in V_f^+$ or $e_n\in V_f^+$. Suppose that $e_{n-1}\in V_f^+$. If $d_{n-1}\in V_f^-$ we obtain a contradiction of Lemma~\ref{Lemma shikama1}(1) considering $e_{n-1},d_{n-1},e$. So $d_{n-1}\notin V_f^-$ and considering the same vertical edge interval we obtain $\ell\in V_f^-$. If $d_n\in V_f^-$ then, considering its vertical edge interval, we have $g\in V_f^+$ or $e_n\in V_f^+$, obtaining in each case a contradiction with Lemma~\ref{Lemma shikama1}(1) considering respectively $g,\ell,e_{n-1}$ and $d_n,e_n,e$. So $d_n\notin V_f^-$. Moreover $e_n\notin V_f^-$, otherwise we have the same contradiction considering the vertices $e_n,e_{n-1},\ell$. Finally if $e_n \in V_f^+$ then, denoting with $\mathcal{I}$ the horizontal edge interval of $e$, we have $|V_f^+ \cap \mathcal{I}|=2$ and $|V_f^- \cap \mathcal{I}|=1$, that is a contradiction. So, $e_n \notin V_f^+$ and we proved points (3) and the case a) of (2). If we assume $e_{n}\in V_f^+$, with the same argument we can prove again points (3) and the case b) of (2). By our proof, it is also easy to argue that the same result holds exchanging $V_{f}^+$ with $V_{f}^-$ and $V_{f}^-$ with $V_{f}^+$.
\end{proof}

\begin{Lemma}
Let $\cP$ be a closed path polyomino and $f\in J_P$ an irredundant binomial with $\deg f \geq 3$. Consider a subpolyomino having a configuration as in Figure~\ref{Fig:Lemma2}.\\
\begin{figure}[h!]
\centering
\includegraphics[scale=0.7]{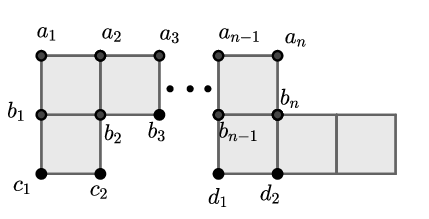}
\caption{Configuration examined in Lemma~\ref{Lemma2}.}
\label{Fig:Lemma2}
\end{figure}

\noindent Suppose that $b_1,b_2\notin V_{f}^-$ and that one of the following occurs:
\begin{itemize}
\item[a1)] $a_{1}\in V_{f}^-$, $c_1 \in V_{f}^+$ and $c_2\notin V_f^+$.
\item[a2)] $a_{2}\in V_{f}^-$, $c_2 \in V_{f}^+$ and $c_1\notin V_f^+$.
\end{itemize}
Then the following hold:
\begin{enumerate}
\item $a_i,b_i \notin V_{f}^+\cup V_{f}^-$ for all $i\in \{3,\ldots,n-2\}$.
\item Exactly one of the following occurs:
\begin{itemize}
\item[a)] $a_{n-1}\in V_{f}^+$, $d_1 \in V_{f}^-$ and $a_n \notin V_{f}^+\cup V_{f}^-$.
\item[b)] $a_{n}\in V_{f}^+$, $d_2 \in V_{f}^-$ and $a_{n-1} \notin V_{f}^+\cup V_{f}^-$.
\end{itemize}
\item $b_1,b_2\notin V_{f}^+$ and $b_{n-1},b_{n} \notin V_{f}^-$.
\item If a1) occurs then $a_2\notin V_{f}^+\cup  V_f^-$, while if a2) occurs then $a_1 \notin V_{f}^+\cup V_f^-$.
\end{enumerate} 
The same result holds exchanging $V_{f}^+$ with $V_{f}^-$ (and vice versa), in all occurrences above.
\label{Lemma2}
\end{Lemma}
\begin{proof}
Suppose a1) holds. Observe that, for $i\in \{3,\ldots,n-2\}$, $a_i\in V_f^+$ if and only if $b_i\in V_f^-$. In particular, if $a_i\in V_f^+$ and $b_i\in V_f^-$ we obtain a contradiction with Lemma~\ref{Lemma shikama1}(1) considering $a_1, a_i, b_i$. So, $a_i\notin V_f^+$ and $b_i\notin V_f^-$ for all $i\in \{3,\ldots,n-1\}$. Also $a_2\notin V_f^+$, otherwise we have a contradiction of Lemma~\ref{Lemma shikama1}(1) considering $a_1,a_2,c_1$. Therefore, considering the horizontal edge interval of $a_1$, either $a_{n-1}\in V_f^+$ or $a_n\in V_f^+$. We assume that $a_{n-1}\in V_f^+$ and we prove that we obtain our result with a) in point (2). Observe that, for $i\in \{3,\ldots,n-2\}$, $a_i\in V_f^-$ if and only if $b_i\in V_f^+$. In particular, $a_i\notin V_f^-$ and $b_i\notin V_f^+$ for all $i\in \{3,\ldots,n-1\}$, otherwise we obtain a contradiction of Lemma~\ref{Lemma shikama1}(1) considering $a_i,b_i,a_{n-1}$. We have also $a_2\notin V_f^-$, otherwise (since $c_2\notin V_f^+$) the only possibility is $b_2\in V_f^+$, obtaining the same contradiction considering $a_2,b_2,a_{n-1}$. As consequence, we obtain also $b_2\notin V_f^+$, otherwise the only possibility is $c_2\in V_f^-$, that is a contradiction of Lemma~\ref{Lemma shikama1}(1) considering $c_1,a_1,c_2$. Moreover $b_1\notin V_f^+$, otherwise we obtain the same contradiction considering $b_1,a_1,a_{n-1}$. Considering the vertical edge interval of $a_{n-1}$, then $b_{n-1}\in V_f^-$ or $d_1\in V_f^-$. In the first case we obtain the same contradiction as before, considering $a_{n-1},b_{n-1},a_1$. So we obtain $d_1\in V_f^-$. As consequence, $a_n\notin V_f^-$, otherwise we obtain the same contradiction considering $a_{n-1},d_1,a_n$. If $a_n\in V_f^+$, denoting with $\mathcal{I}$ the horizontal edge interval of $a_1$, we have $|V_f^+ \cap \mathcal{I}|=2$ and $|V_f^- \cap \mathcal{I}|=1$, that is a contradiction. So, $a_n \notin V_f^+$. Finally, $b_{n-1}\notin V_f^-$ otherwise we have the same contradiction considering $a_{n-1},b_{n-1},a_1$, and $b_n\notin V_f^+$, otherwise the only possibility is $d_2\in V_f^+$ (since $a_n\notin V_f^+$) and we have the same contradiction considering $d_2,d_1,a_{n-1}$. So we obtained our result with a) in point (2). With the same argument we obtain our result with b) in point (2) assuming that $a_{n}\in V_f^+$ (instead of $a_{n-1}\in V_f^+$).\\
If we suppose at the beginning that condition a2) holds, we obtain our result with the same argument. By our proof, it is also easy to argue that the same result holds exchanging $V_{f}^+$ with $V_{f}^-$ and vice versa.
\end{proof}

\begin{Lemma}
Let $\cP$ be a closed path polyomino and $f\in J_P$ an irredundant binomial with $\deg f \geq 3$. Consider a subpolyomino having a configuration as in Figure~\ref{Fig:Lemma3}.

\begin{figure}[h!]
\centering
\includegraphics[scale=0.7]{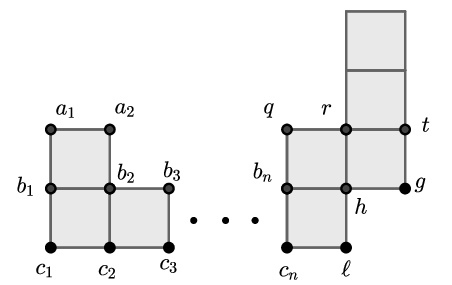}
\caption{Configuration examined in Lemma~\ref{Lemma3}.}
\label{Fig:Lemma3}
\end{figure}

\noindent Suppose that $b_1,b_2\notin V_{f}^-$ and that one of the following occurs:
\begin{itemize}
\item[a1)] $a_{1}\in V_{f}^+$, $c_1 \in V_{f}^-$ and $a_2\notin V_f^+$.
\item[a2)] $a_{2}\in V_{f}^+$, $c_2 \in V_{f}^-$ and $a_1\notin V_f^+$.
\end{itemize}
Then the following hold:
\begin{enumerate}
\item $b_i,c_i \notin V_{f}^+\cup V_{f}^-$ for all $i\in \{3,\ldots,n-1\}$.
\item If a1) occurs then $c_2\notin V_{f}^+\cup  V_f^-$, while if a2) occurs then $c_1 \notin V_{f}^+\cup V_f^-$.
\item $b_n,h\notin V_f^-$ and $b_1,b_2\notin V_f^+$.
\end{enumerate} 
Moreover, if $\cP$ has a zig-zag walk and in addition $g\in V_f^-$, then it holds that $\ell\in V_f^+$, $b_n \notin V_{f}^+$, $c_n,q,r,t\notin V_{f}^+\cup V_f^-$.\\
The same result holds replacing $V_{f}^+$ with $V_{f}^-$ and $V_{f}^-$ with $V_{f}^+$ in all previous occurrences.
\label{Lemma3}
\end{Lemma}
\begin{proof}
Suppose a1) holds. Observe that, for $i\in \{3,\ldots,n-1\}$, $c_i\in V_f^+$ if and only if $b_i\in V_f^-$. In particular, if $c_i\in V_f^+$ and $b_i\in V_f^-$ we obtain a contradiction with Lemma~\ref{Lemma shikama1}(1) considering $c_1, c_i, b_i$. So, $c_i\notin V_f^+$ and $b_i\notin V_f^-$ for all $i\in \{3,\ldots,n-1\}$. Moreover $c_2\notin V_f^+$, otherwise we obtain the same contradiction considering $c_1,a_1,c_2$. It follows that, since $c_1\in V_f^-$, considering the horizontal edge interval of $c_1$ there exists a vertex $X\in V_f^+$ such that $X\in\{c_n,\ell\}$. Therefore, for $i\in \{3,\ldots,n-1\}$, observing that $c_i\in V_f^-$ if and only if $b_i\in V_f^+$, with the same argument above we obtain $c_i\notin V_f^-$ and $b_i\notin V_f^+$ for all $i\in \{3,\ldots,n-1\}$, otherwise we contradict Lemma~\ref{Lemma shikama1}(1) considering $c_i, b_i, X$. Moreover, considering the vertical edge interval of $c_2$, if $c_2\in V_f^-$ then $b_2\in V_f^+$ (since $a_2\notin V_f^+$), that is a contradiction with Lemma~\ref{Lemma shikama1}(1) considering $c_2,b_2,X$. Therefore, at this point, we proved points (1) and (2). 

Suppose $b_2\in V_f^+$, then $a_2\in V_f^-$ (since $c_2\notin V_f^-$), but this is a contradiction with Lemma~\ref{Lemma shikama1}(1) considering $a_1,a_2,b_2$. Moreover also $b_1\notin V_f^+$, otherwise we obtain the same contradiction considering $b_1,c_1,X$. Now suppose $b_n\in V_f^-$. If $c_n\in V_f^+$ we have the same contradiction considering $b_n,c_n,c_1$. Hence, the only possibility is $q\in V_f^+$ but in such case we have also $r\in V_f^-$ or $t\in V_f^-$, obtaining in each case a contradiction with Lemma~\ref{Lemma shikama1}(1). So $b_n\notin V_f^-$. Suppose $h\in V_f^-$. If $\ell\in V_f^+$ we have the same contradiction with Lemma~\ref{Lemma shikama1}(1) considering $h,\ell,c_1$. It follows that the only possibility, for $V_f^+$ in the horizontal edge interval of $c_1$, is  $c_n\in V_f^+$. To avoid the same contradiction, in the vertical edge interval of $c_n$ we have to consider $q\in V_f^-$, and as consequence in the horizontal edge interval of $q$ we have to consider $t\in V_f^+$. But in this case we contradict Lemma~\ref{Lemma shikama1}(2), considering the vertices $c_n,t\in V_f^+$ and $h,c_1\in V_f^-$. So, we obtain $h\notin V_f^-$. At this point, we proved also (3).

Suppose now $g\in V_f^-$ and $\cP$ has a zig-zag walk. By the above arguments, we showed that either $c_n\in V_f^+$ or $\ell\in V_f^+$. If $c_n\in V_f^+$, then either $b_n\in V_f^-$ or $q\in V_f^-$. In the first case we obtain a contradiction with Lemma~\ref{Lemma shikama1}(1) considering $b_n,c_n,c_1$, in the second case we have either $r\in V_f^+$ or $t\in V_f^+$, obtaining again the same contradiction considering respectively $c_n,q,r$ and $q,t,g$. So $c_n\notin V_f^+$ and $\ell\in V_f^+$. Moreover also $c_n\notin V_f^-$, otherwise either $b_n\in V_f^+$ or $q\in V_f^+$, obtaining in each case the same contradiction together with $\ell\in V_f^+$. As consequence also $b_n\notin V_f^+$, otherwise the only possibility is $q\in V_f^-$, obtaining the same contradiction considering $b_n,q,g$. At this point we have $c_n,b_n\notin V_f^+ \cup V_f^-$, and considering their vertical edge interval we obtain also $q\notin V_f^+ \cup V_f^-$. Finally if we suppose the assumption $r,t\notin V_f^+ \cup V_f^-$ is not true, at this point the only possibility is either $r\in V_f^-$ and $t\in V_f^+$, or $r\in V_f^+$ and $t\in V_f^-$. In the first case we obtain a contradiction with Lemma~\ref{Lemma shikama1}(1) considering $r,t,g$. In the second case, consider that there exists a vertex $X$ belonging to the same vertical edge interval of $t$ such that $X\in V_f^+$, and since $\cP$ has a zig-zag walk then $X,t,r$ are vertices of an inner 2-minor of $\cP$, obtaining again a contradiction with Lemma~\ref{Lemma shikama1}(1).\\
We have just proved our results in the case condition a1) holds. It is easy to understand that the same arguments hold also in case condition a2) is satisfied. By our proof, it is also easy to argue that the same result holds exchanging $V_{f}^+$ with $V_{f}^-$ and vice versa.
\end{proof}

\noindent By the following result we can state that the ideal $\pp_1$, defined in the previous section, is a prime ideal.

\begin{theorem}\label{thm:p_1prime}
Let $\cP$ be a closed path polyomino having zig-zag walks and let $Z_\cP$ be the ideal generated by the binomials $f_\cW$, for every zig-zag walk $\cW$ in $\cP$. Then $I_\cP+Z_\cP=J_\cP$.
\end{theorem}
\begin{proof}

$\supseteq)$ Let $f=f^+ - f^-\in J_\cP$ be an irredundant binomial. Since $(J_\cP)_2=I_\cP$, if $\deg f=2$ then $f\in I_\cP$. So, suppose $\deg f\geq 3$. We prove that $f\in Z_\cP$ and in particular $f=\pm f_\cW$ for some zig-zag walk $\cW$ in $\cP$.\\
Consider the unique sub-polyomino of $\cP$ having configuration as in Figure~\ref{fig:threestepladder} and such that the points $a,b,c,d,e$ involve the variable $w$ in the toric ring that define $J_\cP$. We first show that $a,b,d,m\notin V_f^+ \cup V_f^-$. \\
If $b\in V_f^+$ then $a\in V_f^-$ or $m\in V_f^-$. In the first case, there exists a vertex $F\in V_f^+$ in the same vertical edge interval of $a$ and since $\cP$ has zig-zag walks then $a,F,b$ are the corners of an inner 2-minor, that is a contradiction of Lemma~\ref{Lemma shikama1}(1). In the second case, there exists a vertex $F\in V_f^+$ in the same vertical edge interval of $m$ such that $m,F,b$ are the corners of an inner 2-minor, obtaining the same contradiction. So $b\notin V_f^+$ and with the same argument we can prove $b\notin V_f^-$. Suppose that $a\in V_f^+$, since $b\notin V_f^+\cup V_f^-$ then $m\in V_f^-$ and in order to avoid contradictions the only possibility is $e_1\in V_f^+$ and $c,e\notin V_f^-$. It follows that there exists a vertex $F\in V_f^-$ in the same horizontal edge interval of $e$ (with $F\neq e,e_1$), and since $\cP$ has zig-zag walks then $d,e,F$ are the corners of an inner 2-minor. Moreover, considering $a\in V_f^+$, then $w\in \mathrm{supp}(\psi(f^+))=\mathrm{supp}(\psi(f^-))$, and since $b,c,e\notin V_f^-$ the only possibility is $d\in V_f^-$. So, we obtain a contradiction of Lemma~\ref{Lemma shikama1}(2), considering the vertices $a,d,e_1,m,F$. Therefore $a\notin V_f^+$ and with the same argument we can prove $a\notin V_f^-$. Suppose $d\in V_f^+$, then $w\in \mathrm{supp}(\psi(f^+))=\mathrm{supp}(\psi(f^-))$ and since $a,b\notin V_f^+\cup V_f^-$ then $c\in V_f^-$ or $e\in V_f^-$. Since $\cP$ has zig-zag walks, in the first case there exist e vertex $F\in V_f^+$ such that $F$ belong in the same vertical edge interval of $c$ and $c,F,d$ are corners of the same inner 2-minor, while in the second case $F$ belong in the same horizontal edge interval of $e$ and $F,e,d$ are corners of the same inner 2-minor. In both cases we obtain a contradiction of Lemma~\ref{Lemma shikama1}(1). Therefore $d\notin V_f^+$ and with the same argument we can prove $d\notin V_f^-$. Finally, since $a,b\notin V_f^+\cup V_f^-$, considering their horizontal edge interval we obtain $m\notin V_f^+\cup V_f^-$.So, $(V_f^+ \cup V_f^-)\cap \{a,b,c,d,e,m\}\subseteq \{c,e\}$.\\ 
Suppose that  $(V_f^+ \cup V_f^-)\cap \{a,b,c,d,e,m\}=\emptyset$. Then all variables of $f$ are contained in $S'=K[x_v\mid v\in V(\cP)\setminus \{a,b,c,d,e,m\}]$. Consider $\psi'$ the restriction of $\psi$ on $S'$ and $\cP'$ the polyomino obtained by $\cP$ removing the cells in $[c,m]$ and $[e,m]$. It is verified that $f\in \ker \psi'$ and,by the results contained in \cite{QSS}, we have $\ker \psi'=I_{\cP'}$. So $f$ is an irredundant binomial in $I_{\cP'}$, but this means that $\deg f=2$, that contradicts our assumption. So, $(V_f^+ \cup V_f^-)\cap \{c,e\}\neq \emptyset$ and it is not possible $|(V_f^+ \cup V_f^-)\cap \{c,e\}|=1$, since $w \in \mathrm{supp}(\psi(f^+))=\mathrm{supp}(\psi(f^-))$. Then, the only possibility is either $c\in V_f^+$ and $e\in V_f^-$ or $c\in V_f^-$ and $e\in V_f^+$. Without loss of generality we can suppose that the first possibility occurs. Now we continue our argument following the structure of $\cP$. 

%
%
%
%
%

%
\noindent Since $\cP$ has zig-zag walks then (continuing on the ``east'' part) we can continue considering  a subpolyomino having the same shape as one in Figure~\ref{img4_prime}. Let $\cV_1$ be the set of vertices of such a subpolyomino. In particular, considering Figure~\ref{img4_prime-1} and Figure~\ref{img4_prime-2}, we set $i$ varying on $\{1,\ldots,n\}$, where $b_1^{(1)}$ belong to the same horizontal edge interval of $e$ and $n$ is the index such that $a_n^{(1)}$ belongs to the same horizontal edge interval of $c_1$.
\begin{figure}[h]
	\subfloat[]{\includegraphics[scale=0.65]{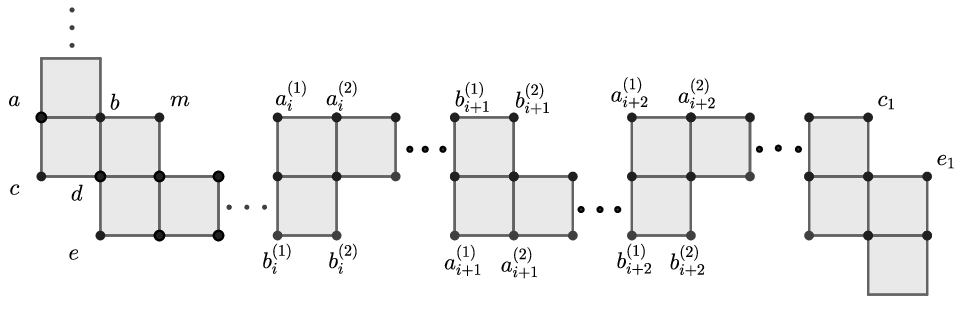}\label{img4_prime-1}}\quad 
	\subfloat[]{\includegraphics[scale=0.65]{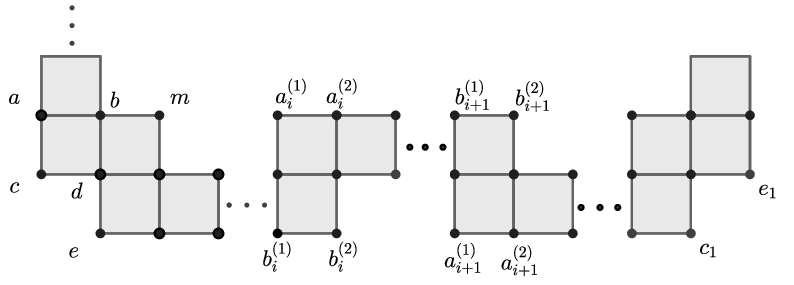}\label{img4_prime-2}}
	\qquad
	\subfloat[]{\includegraphics[scale=0.65]{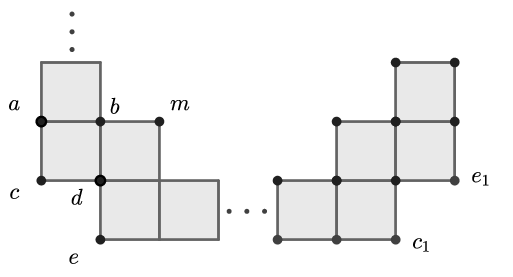}\label{img4_prime-3}}
	\caption{Possible configurations of $\mathcal{C}_k$, up to reflections or rotations.} 
	\label{img4_prime}
\end{figure}

\noindent By Lemma~\ref{Lemma1}, Lemma~\ref{Lemma2} and Lemma~\ref{Lemma3} observe that, in the horizontal edge interval of $c$, all vertices different from $c$ and $e_1$ does not belong to $V^+ \cup V^-$. So, since $c\in V_f^+$, it is verified that $e_1\in V_f^-$. Therefore, by the same lemmas we can argue that $$V^+ \cap \cV_1=\{c,c_1,b_i^{(j_i)} \mid i\in \{1,\ldots,n\}, \mbox{and either $j_{i}=1$ or $j_{i}=2$}\},$$ $$V^- \cap \cV_1=\{e,e_1,a_i^{(j_i)} \mid i\in \{1,\ldots,n\}, \mbox{and either $j_{i}=1$ or $j_{i}=2$}\}.$$\\
Observe that $\cP$ can be built as union of the configurations in Figure~\ref{img4_prime}. In particular we set $\cP=\bigcup_{k=1}^{t} \cC_k$, where $\cC_k$, for all $k\in \{1,\ldots,t\}$, is a configuration as in Figure~\ref{img4_prime-1} or Figure~\ref{img4_prime-2} or Figure~\ref{img4_prime-3}. Let $\cV_k$ the set of the vertices of $\cC_k$ that are highlighted with a black point in the picture.  Denote with $c_k,e_k, a_{i,k}^{(j_i)}, b_{i,k}^{(j_i)}$, for $j_i\in\{1,2\}$, the vertices in $\cC_k$ corresponding to the vertices $c_1,e_1,a_i^{(j_i)},b_i^{(j_i)}$ in the picture, and let $n_k$ be the index such that $a_{n_k,k}^{(1)}$ belong to the same horizontal edge interval of $c_k$. Observe that, since $\cP$ is a closed path, $c_t=c$ and $e_t=e$. Moreover, starting from $c_1\in V_f^+$ and $e_1\in V_f^-$, that we proved before, and considering the vertices in $\cC_1$ not belonging to $V_f^+\cup V_f^-$ (as also we have shown before), by Lemma~\ref{Lemma1}, Lemma~\ref{Lemma2}, Lemma~\ref{Lemma3} and using the same arguments we can obtain that, for all $k\in \{2,\ldots,t\}$:
$$V^+ \cap \cV_k=\{c_{k-1},c_k,b_{i,k}^{(j_i)} \mid i\in \{1,\ldots,n_k\}, \mbox{and either $j_{i}=1$ or $j_{i}=2$}\},$$ $$V^- \cap \cV_k=\{e_{k-1},e_k,a_{i,k}^{(j_i)} \mid i\in \{1,\ldots,n_k\}, \mbox{and either $j_{i}=1$ or $j_{i}=2$}\}.$$

\noindent This means that: 

$$f=\prod_{k=1}^{t}\left(x_{c_k}\prod_{i=1}^{n_k}x_{b_{i,k}^{(j_i)}}\right)-\prod_{k=1}^{t}\left(x_{e_k}\prod_{i=1}^{n_k}x_{a_{i,k}^{(j_i)}}\right)$$

\noindent One can see that (following also the mentioned figures) that the variable involved in $f$ allow to obtain a zig-zag walk $\cW$ of $\cP$ and that the structure of $f$ corresponds exactly to $f=f_\cW$ or $f=-f_\cW$ (the symbol + or - depends on the convention on $z_i$ and $u_i$, with reference to the definition of $f_\cW$ mentioned in the preliminaries of this paper). So, we can conclude $f\in Z_\cP$.\\
\noindent $\subseteq)$ We know that $I_\cP\subseteq J_\cP$. Let $\cW$ be a zig-zag walk and $f_\cW=\prod_{i=1}^{\ell}x_{z_i} - \prod_{i=1}^{\ell}x_{u_i}$ the related binomial. Set $f^+ =\prod_{i=1}^{\ell}x_{z_i}$ and $f^- =\prod_{i=1}^{\ell}x_{u_i}$, we prove that $\psi(f^+)= \psi(f^-)$.  Let $j\in\{1,\ldots,\ell\}$, $x_{z_j}\in \mathrm{supp}(f^+)$ and  $V_j$ and $H_j$ be respectively the vertical and horizontal edge interval of $z_j$, so $v_j,h_j \in \mathrm{supp}(\psi(f^+))$. We prove that $v_j,h_j \in \mathrm{supp}(\psi(f^-))$. Consider that $\cP$ can be built as union of the configurations in Figure~\ref{img4_prime}. Using the same construction of the previous part of the proof, we consider $\cP=\bigcup_{k=1}^{t} \cC_k$ where $\cC_k$, for all $k\in \{1,\ldots,t\}$, is a configuration as in Figure~\ref{img4_prime-1}, Figure~\ref{img4_prime-2} or Figure~\ref{img4_prime-3}, and denote with $c_k,e_k, a_{i,k}^{(j_i)}, b_{i,k}^{(j_i)}$, for $j_i\in\{1,2\}$, the vertices in $\cC_k$ corresponding to the vertices $c_1,e_1,a_i^{(j_i)},b_i^{(j_i)}$. Let $n_k$ be the index such that $a_{n_k,k}^{(1)}$ belong to the same horizontal edge interval of $c_k$. It is easy to see that $z_j\in \{c_k,e_k, a_{i,k}^{(j_i)}, b_{i,k}^{(j_i)}\}$ for some $k\in \{1,\ldots,t\}$. In each case we obtain that $v_j,h_j\in \mathrm{supp}(\psi(f^-))$, as we explain in the following: 

\begin{itemize}
\item if $z_j=c_k$, we can consider  $a_{n_k,k}^{(j_i)}, e_{k+1} \in f^-$, for $j_i=1$ or $j_i=2$, with either $a_{n_k,k}^{(j_i)}\in V_i$ and $e_{k+1}\in H_j$ or $a_{n_k,k}^{(j_i)}\in H_j$ and $e_{k+1}\in V_j$; 
\item if $z_j=e_k$, we can consider  $b_{1,k}^{(j_i)}, c_{k-1} \in f^-$, for $j_i=1$ or $j_i=2$, with either $b_{1,k}^{(j_i)}\in V_j$ and $c_{k-1}\in H_j$ or $b_{1,k}^{(j_i)}\in H_j$ and $c_{k-1}\in V_j$;
\item if $z_j=a_{i,k}^{(j_i)}$ with $i\neq n_k$, we can consider  $b_{i,k}^{(j_i)}, b_{i+1,k}^{(j_{i+1})}\in f^-$ with either $b_{i,k}^{(j_i)}\in V_j$ and $b_{i+1,k}^{(j_{i+1})}\in H_j$ or $b_{i,k}^{(j_i)}\in H_j$ and $b_{i+1,k}^{(j_{i+1})}\in V_j$; \item if $z_j=a_{n_k,k}^{(j_i)}$ we can consider  $b_{n_k,k}^{(j_i)}, c_k\in f^-$ with either $b_{n_k,k}^{(j_i)}\in V_j$ and $c_k\in H_j$ or $b_{n_k,k}^{(j_i)}\in H_j$ and $c_k\in V_j$; 
\item if $z_j=b_{i,k}^{(j_i)}$ with $i\neq 1$, we can consider  $a_{i,k}^{(j_i)}, a_{i-1,k}^{(j_{i-1})}\in f^-$ with either $a_{i,k}^{(j_i)}\in V_j$ and $a_{i-1,k}^{(j_{i-1})}\in H_j$ or $a_{i,k}^{(j_i)}\in H_j$ and $a_{i-1,k}^{(j_{i-1})}\in V_j$;
\item if $z_j=b_{1,k}^{(j_i)}$ we can consider  $a_{1,k}^{(j_i)}, e_{k-1}\in f^-$ with either $a_{1,k}^{(j_i)}\in V_j$ and $e_{k-1}\in H_j$ or $a_{1,k}^{(j_i)}\in H_j$ and $e_{k_1}\in V_j$.
\end{itemize}
 Finally , with reference to Figure~\ref{fig:threestepladder} we easily obtain that either $x_c\in \mathrm{supp}(f^+)$ and $x_e\in \mathrm{supp}(f^-)$ or $x_c\in \mathrm{supp}(f^-)$ and $x_e\in \mathrm{supp}(f^+)$. So we can conclude that $\psi(f^+)=w\prod_{i=1}^{\ell}v_i h_i = \psi(f^-)$, that is $f_\cW\in J_\cP$.
\end{proof}

\noindent So we have proved that the ideal $\mathfrak{p}_1$, defined in the previous section, is prime. By the following general result we can also argue that $\mathrm{height}(\mathfrak{p}_1)=|\cP|$

\begin{proposition}\label{prop:p_1prime}
Let $\cP$ be a collection of cells having zig-zag walks and let $Z_\cP$ be the ideal generated by the binomials $f_\cW$, for every zig-zag walk $\cW$ in $\cP$. Denoted $J=I_\cP+Z_\cP$,
then $\mathrm{height}(J)\leq |\cP|$. Moreover if $J$ is unmixed or $\mathrm{height}(I_\cP)=|\cP|$ then equality holds.
\end{proposition}
\begin{proof}
Let $V_{J}$ be the $\mathbb{Q}$-vector space generated by the set $\{\mathbf{v}-\mathbf{w}\in \mathbb{Q}^n\mid x^\mathbf{v}-x^\mathbf{w}\in J\}$. By \cite[Theorem 1.1]{height} then $\mathrm{height}(J)\leq \dim_\mathbb{Q}V_{J}$, and if $J$ is unmixed equality holds. We prove that $\dim_\mathbb{Q}V_{J}=|\cP|$. Observe that $V_{J}$ is a subspace of $\mathbb{Q}^{|V(\cP)|}$. For $a\in V(\cP)$, we denote by $\mathbf{v}_a$ the vector in $\mathbb{Q}^{|V(\cP)|}$  whose $a$-th component is 1, while its other components are 0 (in particular, the set $\{\mathbf{v}_a\mid a\in V(\cP)\}$ is the canonical basis of $\mathbb{Q}^{|V(\cP)|}$). Moreover if $I=[a,b]$ is an inner interval of $\cP$, with diagonal corners $a,b$ and anti-diagonal corners $c,d$, we denote $\mathbf{v}_I=\mathbf{v}_a+\mathbf{v}_b-\mathbf{v}_c-\mathbf{v}_d$. Let $\cB=\{\mathbf{v}_C\mid C\ \mbox{is a cell of}\ \cP\}$, by the proof of \cite[Theorem 3.1]{height} we know that the vectors in $\cB$ are linearly independent and if $I$ is an inner interval of $\cP$ then $\mathbf{v}_I\in \langle\cB\rangle$. Moreover, if $f_\cW=\prod_{i=1}^{l}x_{z_i} - \prod_{i=1}^{l}x_{u_i}$ is the binomial associated to a zig-zag walk $\cW:I_1,\ldots,I_l$, it is not difficult to check that $\sum_{i=1}^r \mathbf{v}_{z_i}-\sum_{i=1}^r \mathbf{v}_{u_i}=\sum_{i=1}^r (-1)^{i+1}\mathbf{v}_{I_i}\in \langle \cB\rangle$.
So $V_{J}=\langle \cB \rangle$ and in particular $\dim_\mathbb{Q} V_{J}=|\cP|$. Finally, if $|\cP|=\mathrm{height}(I_\cP)$, then $|\cP|=\mathrm{height}(I_\cP)\leq \mathrm{height}(J)\leq |\cP|$. Therefore $\mathrm{height}(J)=|\cP|$.
\end{proof}

\subsection{Final result}


\begin{Lemma}\label{lem:zig-zagBinomial}
       Let $\MP$ be a non-prime closed path polyomino and let $\pp$ be a minimal prime of $I_{\MP}$. 
       If $f_\MW\notin \pp$ for some zig-zag walk $\MW$ of $\MP$, then $\pp=\pp_2$, where $\pp_2$ is as defined in Notation~\ref{not:primeideals}.
   \end{Lemma}
   \begin{proof}
       Suppose $f_\cW=\prod_{i=1}^{l}x_{z_i} - \prod_{i=1}^{l}x_{u_i}\notin \pp$ for some $\cW$. 
       First we show that under the labelling of Figure~\ref{fig:threestepladder},  $x_m \in \pp$. 
       Note that $x_{m}x_e - x_{b}x_{e_1}, x_{m}x_c - x_{a}x_{d_1}\in \pp$. Since, by Lemma~\ref{lem:onezigzagwalk} and Remark~\ref{necklace}, $x_{b}, x_{d_1} \in \pp$, we get $x_m x_e, x_m x_c \in \pp$. 
       If $x_m\notin \pp$ we obtain that $x_e,x_c\in \pp$, and consequently $f_\cW\in \pp$ by~\eqref{rem:zigzagbinomial} of Remark~\ref{necklace}, which is a contradiction. 
       Therefore $x_m\in \pp$.
       Under the assumption, by Lemma~\ref{lem:onezigzagwalk}, $x_v\in \pp$ for all $v\in N(\MP)$. Therefore, we have $\pp_2 \subseteq \pp$.
       Since $\pp$ is a minimal prime of $I_\MP$, $\pp_2$ is a prime ideal (see Lemma~\ref{prop:p_2prime}) and $I_\cP \subseteq p_2$ by Lemma~\ref{lem:i_pinp2}, we get $\pp=\pp_2$.
       Hence the proof.
   \end{proof}
 
   \noindent We are now ready to prove our main theorem. We recall that an ideal $I$ of $K[x_1,\dots,x_n]$ is called \textit{unmixed} if all associated prime ideals of $I$ have the same height.

   \begin{theorem}\label{thm:mainprimary}
        Let $\MP$ be a non-prime closed path polyomino.
        Then $I_\MP = \pp_1 \cap \pp_2$, where $\pp_1$ and $\pp_2$ are as defined in Notation~\ref{not:primeideals}. In particular, $I_\MP$ is unmixed. 
   \end{theorem}

   \begin{proof}
       First, note that $\pp_1$ and $\pp_2$ are prime ideals by Proposition~\ref{prop:p_2prime} and Theorem~\ref{thm:p_1prime}. 
       Clearly, $I_\MP\subset \pp_1$ and by Lemma~\ref{lem:i_pinp2}, $I_\MP\subset \pp_2$. 
   Since $\mathrm{height}(\pp_1)=\mathrm{height}(\pp_2) =\mathrm{height}(I_\MP)$ by Propositions~\ref{prop:I_Pheight},~\ref{prop:p_2prime} and~\ref{prop:p_1prime}, 
       $\pp_1$ and $\pp_2$ are  minimal prime ideals containing $I_\MP$. 
       Moreover, by \cite{Cisto_Navarra_groebner} there exists a Gr\"{o}bner basis of $I_{\MP}$ with square-free initial terms, so $I_\MP$ is radical (see the proof of \cite[Corollary 2.2]{binomial_edge}). 
       Therefore, every associated prime of $I_{\MP}$ is minimal. 
       So it suffices to show that $\pp_1$ and $\pp_2$ are the only minimal prime ideals containing $I_\MP$.  
       Let $\pp$ be a minimal prime of $I_\MP$. 
       By~\cite[Theorem\ 6.2]{Cisto_Navarra_closed_path}, $\MP$ contain zig-zag walks.
       If there exists a zig-zag walk $\MW$ of $\MP$ such that $f_\MW \notin \pp$, then $\pp=\pp_2$ by Lemma~\ref{lem:zig-zagBinomial}.
       If for all zig-zag walk $\MW$ of $\MP$, $f_\MW \in \pp$, then $\pp_1\subseteq \pp$.
       Since $\pp_1$ is a prime ideal, we get that $\pp=\pp_1$.
Hence the proof.
   \end{proof}
   
   

   \section{Insights and open questions}\label{sec:questions}
    
Let $\cP$ be a closed path polyomino. Comparing Theorem~\ref{thm:mainprimary} with Theorem~\ref{primary-total} we can observe that $\mathfrak{p}_1=J_\emptyset =L_\cP$. Moreover $Y=N(\MP)\cup M(\MP)$ is an admissible set of $\cP$ and $\pp_2 =J_Y$. In such a particular case the polyocollection $\cP^{(Y)}$ consists of the disjoint intervals related to the binomials in $R(\cP)$. See for instance the closed path in Figure~\ref{fig:polyomino-polyocoll}, where the vertices in the set $Y$ are highlighted in green and $\cP^{(Y)}$ consists of the red intervals.  

\begin{figure}[h!]
\includegraphics[scale=0.70]{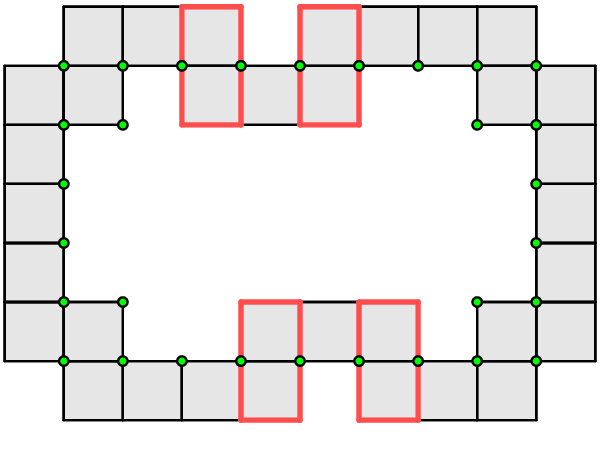}
\caption{A polyomino $\cP$, with the admissible set $Y$ in green and in red the polyocollection $\cP^{(Y)}$} 
	\label{fig:polyomino-polyocoll}
\end{figure}

Observe that not all admissible sets $X$ of $\cP$ are related to a minimal prime of $I_\cP$. In fact, consider in Figure~\ref{fig:polyomino-polyocoll2-1} the same polyomino with the highlighted vertices $a$ and $b$. Then the set $\{a,b\}$ is an admissible set of $\cP$, but $J_{\{a,b\}}$ is not a minimal prime of $I_\cP$, since $J_\emptyset \subsetneq J_{\{a,b\}}$. The related polyocollection $\cP^{(\{a,b\})}$ is pictured in Figure~\ref{fig:polyomino-polyocoll2-2}. 

      
\begin{figure}[h]
	\subfloat[]{\includegraphics[scale=0.65]{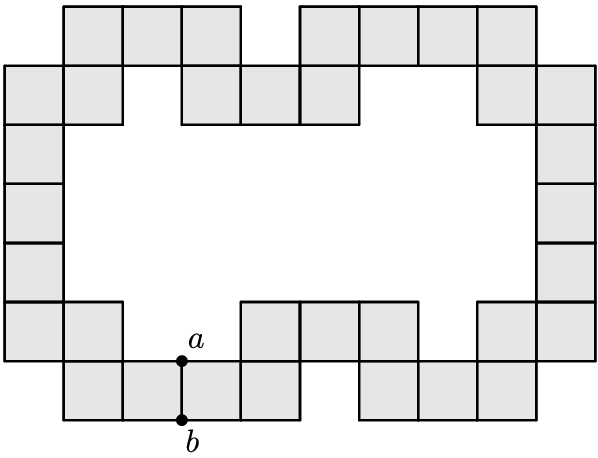}\label{fig:polyomino-polyocoll2-1}}\qquad 
	\subfloat[]{\includegraphics[scale=0.65]{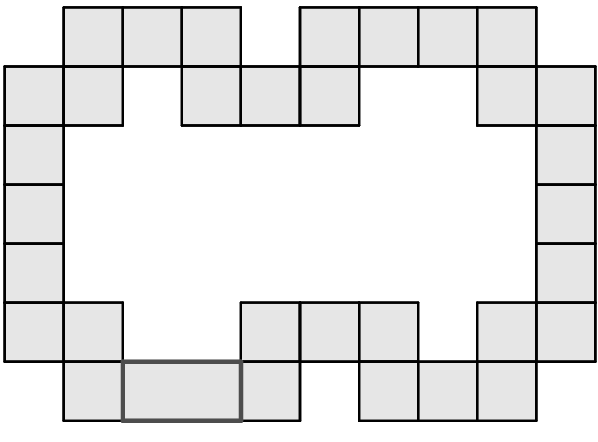}\label{fig:polyomino-polyocoll2-2}}
	\caption{A polyomino $\cP$ and the polyocollection $\cP^{(\{a,b\})}$} 
	\label{fig:polyomino-polyocoll2}
\end{figure}    
    
Some questions arise from the results obtained in this work:

\begin{itemize}

\item[(a)] If $\cP$ is a collection of cells it is well know that some properties of the ideal $I_\cP$ are related to the combinatorics of $\cP$. For instance it is well known that if $\cP$ is simple then $I_\cP$ is a normal Cohen-Macaulay domain of dimension $|V(\cP)|-|\cP|$ (see \cite{coordinate, QSS}). Also the Gorenstein property is known for some classes of polyominoes (see \cite{Andrei, hilbert-closed-path, QUR12, hilbert_parallelogram, thin}). Moreover, referring to Figure \ref{fig:polyomino-polyocoll2}, note that $I_{\cP^{(\{a,b\})}}$ is not prime. In fact, if $\cW$ is a zig-zag walk of the polyomino in Figure~\ref{fig:polyomino-polyocoll2-1}, it is not difficult to see that the binomial $f_\cW$ is a zero divisor of $I_{\cP^{(\{a,b\})}}$ (by the same arguments explained at the beginning of Section~\ref{sec:closedpath}). Furthermore, $I_{\cP^{(\{a,b\})}}$ is radical, since the set of generators forms the reduced Gr\"obner basis with respect to a suitable order obtained as done in \cite{Cisto_Navarra_groebner}. In light of the previous considerations, we wonder if some results holding for polyominoes can be extended also for polyocollections.

\item[(b)] For a collection of cells $\cP$, it is an open question whether the ideal $I_\cP$ is always radical. If this is true, then Theorem~\ref{primary-total} always provides a primary decomposition of it. Such a question can be extended naturally for the ideal $I_\cC$ related to a polyocollection. Another property that could be investigated is the unmixedness for the ideal related to a polyomino (and in general for a polyocollection), that we find for closed paths.

\item[(c)] In order to have a better description of the primary decomposition provided in Theorem~\ref{primary-total}, we should have a better understanding of the ideal $L_\cC$ for a polyocollection $\cC$. For a closed path $\cP$ we have proved that $L_\cP=I_\cP+Z_\cP$, where $Z_\cP$ is the ideal generated by the binomials $f_\cW$ related to zig-zag walks. Anyway we know that this is not true for all polyominoes, as shown in \cite[Example 3.8]{MRR20primalityOfPolyominoes}. In general, for a polyomino $\cP$, $L_\cP=I_\cP+T_\cP$ where $T_\cP$ is a binomial ideal and if $\cW$ is a zig-zag walk then $f_\cW\in T_\cP$. It would be interesting to know the (combinatorial) structure of the generators of $T_\cP$ for other classes of polyominoes (or polyocollections) $\cP$.

\item[(d)] For a closed path $\cP$, we find that the minimal prime ideal of $I_\cP$, different to $L_\cP$, is related to an admissible set containing the (unique) necklace of all zig-zag walks. Anyway, in general a polyomino with zig-zag walks has different necklaces related to different zig-zag walks. The concept of zig-zag walk can be formulated also for polyocollections. In general, if $\cC$ is a polyocollection, we ask if in the primary decomposition provided in Theorem~\ref{primary-total}, the minimal primes of $I_\cC$ different to $L_\cC$ are ideals of kind $J_X$ where $X$ is related in some way to a necklace of a zig-zag walk. Examples with \texttt{Macaulay2} suggest such a behavior for some polyomino ideals.
\end{itemize}  

\subsection*{Acknowledgments} The first and the second author wish to thank Giancarlo Rinaldo for inspiring the studying of the primary decomposition of closed path polyominoes. 
All the authors wish to thank J\"urgen Herzog, who encouraged them to study the primary decomposition for polyominoes in the ``EMS Summer School of Combinatorial Commutative Algebra'' held in Gebze (Turkey) during the SCALE conference, where the authors met. 
The first author want to thank also the group GNSAGA of INDAM, that partially supported by a grant his participation to the conference.
The last author was supported by a grant of IIT Gandhinagar and was partly supported by an Infosys Foundation fellowship. 

We thank Guo Jin for the suggestion of fixing our definition of inner interval of a polyocollection (specifically, that the interiors must be disjoint) and for pointing out a counterexample for the original Proposition 3.2. In this revised version, we have addressed these issues by removing Proposition 3.2 and explicitly adding the disjoint interiors condition.

\end{document}